\documentclass[12pt]{article}
\usepackage[utf8]{inputenc}
\usepackage{amstext,amsfonts,amsthm,amsmath,amssymb}
\usepackage{graphicx,tikz}
\usepackage{enumitem}
\usepackage{xcolor}
\usepackage[colorlinks=true,citecolor=red]{hyperref}
\usepackage[capitalize]{cleveref}
\usepackage[vietnamese,english]{babel}

\newtheorem{theorem}{Theorem}

\newtheorem{lemma}[theorem]{Lemma}
\newtheorem{claim}[theorem]{Claim}
\newtheorem{corollary}[theorem]{Corollary}

\newtheorem{conjecture}{Conjecture}
\newtheorem{problem}{Problem}

\DeclareMathOperator{\sat}{sat*}
\newcommand{\floor}[1]{\left\lfloor #1 \right\rfloor}
\newcommand{\ceil}[1]{\left\lceil #1 \right\rceil}
\newcommand{\A}{\mathcal{A}}
\newcommand{\B}{\mathcal{B}}
\newcommand{\F}{\mathcal{F}}
\newcommand{\C}{\mathcal{C}}
\newcommand{\I}{\mathcal{I}}
\newcommand{\Q}{\mathcal{Q}}
\newcommand{\T}{\mathcal{T}}
\renewcommand{\S}{\mathcal{S}}
\newcommand{\R}{\mathcal{R}}
\newcommand{\D}{\mathcal{D}}
\newcommand{\X}{\mathcal{X}}

\title{Exact antichain saturation numbers via a generalisation of a result of Lehman-Ron}

\author{Paul Bastide\footnote{ENS Rennes, \href{mailto:paul.bastide@ens-rennes.fr}{paul.bastide@ens-rennes.fr}.} \quad Carla Groenland\footnote{Utrecht University, \href{mailto:c.e.groenland@uu.nl}{c.e.groenland@uu.nl}. This project has received funding from the European Union’s Horizon 2020 research and innovation programme under the ERC grant CRACKNP (number 853234) and the Marie Skłodowska-Curie grant GRAPHCOSY (number 101063180). Views and opinions expressed are however those of the author(s) only.}\quad Hugo Jacob\footnote{ENS Paris-Saclay, \href{mailto:hjacob@ens-paris-saclay.fr}{hjacob@ens-paris-saclay.fr}.}\quad Tom Johnston\footnote{University of Bristol  and Heilbronn Institute for Mathematical Research, \href{mailto:tom.johnston@bristol.ac.uk}{tom.johnston@bristol.ac.uk}.}}

\date{\today}

\begin{document}

\maketitle
\begin{abstract}
    For given positive integers $k$ and $n$, a family $\mathcal{F}$ of subsets of $\{1,\dots,n\}$ is $k$-antichain saturated if it does not contain an antichain of size $k$, but adding any set to $\mathcal{F}$ creates an antichain of size $k$.
    We use sat$^*(n, k)$ to denote the smallest size of such a family. For all $k$ and sufficiently large $n$, we determine the exact value of sat$^*(n, k)$. Our result implies that sat$^*(n, k)=n(k-1)-\Theta(k\log k)$, which confirms several conjectures on antichain saturation. Previously, exact values for sat$^*(n,k)$ were only known for $k$ up to $6$.

    We also prove a strengthening of a result of Lehman-Ron which may be of independent interest. We show that given $m$ disjoint chains in the Boolean lattice, we can create $m$ disjoint skipless chains that cover the same elements (where we call a chain skipless if any two consecutive elements differ in size by exactly one).
\end{abstract}
\paragraph{Keywords} 
\section{Introduction}
Many powerful results have been proved over the years concerning the structure of chains and antichains in the Boolean lattice, e.g. \cite{Griggs82,Kleitman75,Saks79,Sperner28,Sudakov22}. 
For example, it is well-known that the Boolean lattice admits a symmetric chain decomposition \cite{aigner,GK76}, and in fact these chains may be taken to be \textit{skipless} (or \textit{saturated}): every chain $C_1\subsetneq \dots\subsetneq C_r\subseteq [n]=\{1,\dots,n\}$  has the property that $|C_{i+1}|=|C_{i}|+1$ for all $i\in [r-1]$. Skipless chains have also been studied in other contexts such as in \cite{Bajnok96,Duffus19,Logan02}.

Given sets $X_1, \dots, X_m$ from layer $r$ and sets $Y_1, \dots, Y_m$ from layer $s$ such that $X_i \subseteq Y_i$, it need not be possible to find disjoint skipless chains $C^1, \dots, C^m$ linking $X_1$ to $Y_1$, $X_2$ to $Y_2$ etc. However, it was shown by Lehman and Ron \cite{lehman2001disjoint} in 2001 that there always exist $m$ disjoint skipless chains that \emph{cover} the sets $X_1, \dots, X_m$ and  $Y_1, \dots, Y_m$.
\begin{theorem}[Lehman-Ron \cite{lehman2001disjoint}]
Let integers $1\leq s<r\leq n$ and subsets $X_1,\dots,X_m,Y_1,\dots,Y_m\subseteq [n]$ be given with $|X_i|=s,$ $|Y_i|=r$ and $X_i\subseteq Y_i$ for all $i\in [m]$. Then there exist $m$ disjoint skipless chains that cover $\{X_1,\dots,X_m,Y_1,\dots,Y_m\}$.
\end{theorem}
It is natural to ask if a stronger statement holds. For example, what happens if we allow the sets to come from different layers, or ask that the chains go via some elements from layers between layer $r$ and layer $s$? Is it possible to cover any $m$ disjoint chains with $m$ disjoint skipless chains, or can we force the use of an additional chain? We show that $m$ chains always suffice. 
\begin{theorem}
    \label{thm:gener}
    Suppose that $\F\subseteq 2^{[n]}$ admits a chain decomposition into $m$ chains. Then there exist disjoint skipless chains $C^1,\dots,C^m$ such that $\F\subseteq \bigcup_{i=1}^m C^i$.
\end{theorem}
This was already known in the special case that the union $\F$ of the chains we wish to cover is a convex set system (i.e. if $X, Y \in \F$ and $X \subseteq Z \subseteq Y$, then $Z \in \F$) \cite{Duffus19}. In this case, the chains can be taken to partition $\F$ as any additional sets must be at the ends of the chains. 

Although we believe Theorem \ref{thm:gener} to be of interest in its own right, our initial motivation came from the area of induced poset saturation where we use Theorem \ref{thm:gener} to easily settle various conjectures concerning the asymptotics of antichain saturation numbers. With more work, we are in fact able to go well beyond the conjectures and pinpoint the exact values. 

For given positive integers $k$ and $n$, a family $\F$ of subsets of $[n]$ is \textit{$k$-antichain saturated} if it
does not contain an antichain of size $k$, but for all $X\subseteq [n]$ with $X\not \in \F$, the family $\F\cup \{X\}$ does contain an antichain of size $k$. We denote the size of the smallest such family by $\sat(n, k)$.

In the literature, this is also sometimes denoted $\sat(n, \A_k)$, where $\A_k$ is the poset consisting of $k$ incomparable elements. This is called an \emph{induced} saturation number: it is the size of the smallest set system which is saturated in terms of not containing $\A_k$ as an \emph{induced} subposet. Such saturation numbers for the Boolean lattice were introduced by Gerbner, Keszegh, Lemons, Palmer, P{\'a}lv{\"o}lgyi
and Patk{\'o}s \cite{gerbner2013saturating} and have been investigated for a variety of posets, for example for the butterfly \cite{ivanbutterfly}, the diamond \cite{ivandiamond} and the chain \cite{scott}. We refer to \cite{KESZEGH2021105497} for a nice overview.

Ferrara, Kay, Kramer, Martin, Reiniger, Smith and Sullivan \cite{ferrara2017saturation} were the first to study the particular case of the antichain and made the following conjecture.
\begin{conjecture}[\cite{ferrara2017saturation}]
    \label{conj:1}
    For $k \geq 3$, $\sat(n, k) \sim (k-1)n$ as $n \to \infty$.
\end{conjecture}
The upper bound is easy to see: for all $i\in [n]$, a $k$-antichain saturated family can contain at most $k-1$ subsets of size $i$ since two subsets of the same size are incomparable. Moreover, a $k$-antichain saturated family must always exist since we can start with the empty family and greedily add subsets until it is no longer possible to do so without creating an antichain of size $k$.

Martin, Smith and Walker \cite{martin2019improved} proved the lower bound
\[ \sat(n,k)\geq \left(1-\frac{1}{\log_2(k-1)} \right) \frac{(k-1)n}{\log_2(k-1)}\]
for $k \geq 4$ and $n$ sufficiently large. The exact values for $k = 2, 3$ and $4$ were shown to be $n + 1$, $2n$ and $3n -1$ respectively in \cite{ferrara2017saturation}, the exact values for $k=5$ and $k=6$ were recently determined to be $4n -2$ and $5n -5$ respectively by {\DJ}ankovi{\'c} and Ivan \cite{djankovic2022saturation}. They also strengthened Conjecture \ref{conj:1} as follows, and proposed two weaker conjectures implied by this conjecture.
\begin{conjecture}[\cite{djankovic2022saturation}]
    \label{conj:2}
    $\sat(n, k) = n(k -1) - O_k(1)$.
\end{conjecture}
At the end of this introduction, we will show how all the conjectures mentioned above are an easy corollary of Theorem \ref{thm:gener}.
\begin{corollary}
\label{cor:1}
There exist constants $c_1,c_2>0$ such that for all $k\geq 4$ and $n$ sufficiently large,
\[
    n(k-1)-c_1 k\log k\leq \sat(n,k)\leq n(k-1)-c_2 k\log k.
\]
\end{corollary}
In general, obtaining exact saturation numbers is a notoriously difficult problem, and for the antichain exact numbers were only known for $k$ up to $6$. 
Our main result determines the exact value of $\sat(n,k)$ for all values of $k$ and $n$ where $n$ is large enough relative to $k$ . We note that $n$ need not be excessively large compared to $k$ and it certainly suffices to assume $n \geq 6 \log k + 1$ for example. Determining the exact values is considerably more involved than just determining the asymptotics, and we require some more definitions just to state the value of the numbers.

Given a natural number $k$, let $\ell$ be the smallest integer $j$ such that $\binom{j}{\floor{j/2}} \geq k - 1$. Note that when $n< \ell$, there are no antichains of size $k$ in $2^{[n]}$ and $\mathcal{F}$ must contain every set (i.e. $\sat(n,k) = 2^n$).

Let $\C(m,t)$ denote the initial segment of layer $t$ of size $m$ when the sets are in colexicographic order. For a family of sets $\A$ from the same layer, let $\nu(\A)$ be the size of the maximum matching from $\A$ to its shadow $\partial \A$, and recursively define $c_0,c_{1}, \dots, c_{\floor{\ell/2}}$ as follows.
Let $c_{\floor{\ell/2}} =k-1$. For $0\leq t<\floor{\ell/2}$, let $c_t = \nu\left(\C(c_{t+1},t+1)\right)$.
\begin{theorem}
    \label{thm:main}
    Let $n, k \geq 4$ be integers and let $\ell$ and $c_0,\dots,c_{\floor{\ell/2}}$ be as defined above. If $n< \ell$, then $\sat(n,k)=2^n$.
    If $n\geq \ell$, then
    \[
        \sat(n,k) \geq  2 \sum_{t=0}^{\floor{\ell/2}} c_t+(k-1)(n - 1 - 2\floor{\ell/2}).
    \]
    Moreover, equality holds when $n\geq 2\ell+1$.
\end{theorem}
Given the form of the bound in Theorem \ref{thm:main}, one might be tempted to suggest that the best approach is to take each layer $t \leq \floor{\ell/2}$ to be an initial segment of colex of the appropriate size, but this is not the case in general. While such an example would have the optimal size, it may already contain an antichain of size $k$. For example, one can check there is an antichain of size 262 in $\C(261, 5) \cup \C(219, 4)$, and this approach would not work for $k = 262$.

For infinitely many values of $k$, a matching upper bound to Theorem \ref{thm:main} was already known \cite{ferrara2017saturation} (see Section \ref{sec:upper_cor}) which works for all $n \geq \ell + 1$. It gives the following corollary.
\begin{corollary}
    \label{cor:nicecase}
    Let $\ell,k,n$ be integers such that $\binom{\ell}{\lfloor \ell/2\rfloor}=k-1$. If $n\leq \ell$ then $\sat(n,k)=2^n$. If $n\geq \ell+1$, then
    \[
        \sat(n,k)=2 \sum_{j=0}^{\lfloor \ell/2\rfloor} \binom{\ell}j + (k-1)(n-1-2\lfloor \ell/2\rfloor).
    \]
\end{corollary}
In particular, whenever $k - 1$ is a central binomial coefficient (i.e. $k=3,4,7,11,21,36,\dots$) the value of $\sat(n,k)$ is determined for all $n$.

\medskip 
We now explain how Corollary \ref{cor:1} follows from Theorem \ref{thm:gener}. The upper bound was already known, and we prove a lower bound of $\sat(n,k)\geq (n+1-2\ell)(k-1)$ for $n$ sufficiently large. (Recall that $\ell$ is the smallest $j$ such that $\binom{j}{\lfloor j/2\rfloor}\geq k-1$, so $\ell=\Theta(\log k)$.)

By Dilworth's theorem \cite{dilworth1950decomposition}, having a chain decomposition of size at most $k-1$ is equivalent to not containing any antichain of size $k$. 
Suppose that $\F\subseteq 2^{[n]}$ is $k$-antichain saturated and so admits a decomposition into $k-1$ chains.
By Theorem \ref{thm:gener}, there are $k-1$ disjoint skipless chains $C^1,\dots,C^{k-1}$ that cover the elements of $\F$; since $\F$ is saturated, this must form a chain decomposition of $\F$. It suffices to show that every chain must contain a set of size at most $\ell$ and a set of size at least $n- \ell$. Suppose the smallest element $X$ of some chain $C^i$ has size $|X|>\ell$, then all subsets $Y$ of $X$ must be present in $\F$ since otherwise we may extend $C^i$ to include $Y$ (and that would mean that $\F\cup \{Y\}$ can also be covered by $k-1$ chains, contradicting the fact that $\F$ is $k$-antichain saturated). There are at least $k-1$ subsets of $X$ of size $\floor{\ell/2}$, and these cannot all be covered by the other $k-2$ chains. Since each chain contains an element of size at most $\ell$ and one of size at least $n-\ell$, the bound follows immediately from the fact that the chains are skipless.

\medskip
In order to prove the exact lower bound of Theorem \ref{thm:main}, we need to examine what happens on layers $1,\dots,\ell$. This is considerably more delicate and for this we use an auxiliary result concerning the matching number of the colex order (Lemma \ref{lem:colex_best}), which we give in Section \ref{sec:shadow_match}.
Finally, we give an explicit construction of a $k$-antichain saturated system $\F$ which matches our lower bound on each layer provided $n$ is sufficiently large. This construction was already known for the special case $k-1 = {\ell \choose \lfloor \ell/2\rfloor}$, and we apply it recursively for other values of $k$. The recursion requires special care and depends on a particular way of writing $k-1$ as a sum of binomial coefficients. This notation can be used to write exact values for the matching numbers $c_t$ from Theorem \ref{thm:main} (see Section \ref{sec:cascade}).

In Section \ref{sec:prelim}, we introduce our notation and the auxiliary results. In Section \ref{sec:gener} we prove Theorem \ref{thm:gener}. In Section \ref{sec:proof} we give the proofs of the lower bounds of Theorem \ref{thm:main} and \cref{cor:nicecase}. In Section \ref{sec:upperbound} we finish the proofs of Theorem \ref{thm:main}  and \cref{cor:nicecase} by giving the upper bound constructions.
In Section \ref{sec:concl} we give directions for future work.

\section{Preliminaries}
\label{sec:prelim}
Let $G=(U,V,E)$ be a bipartite graph on vertex sets $U$ and $V$ with edge set $E$. For $X\subseteq U$, we write $N(X)$ for the set of neighbours of $X$.
A \emph{matching} $M$ between $U$ and $V$ is a set of edges $M\subseteq E$ such that the edges are pairwise disjoint (i.e. $m\cap m'=\emptyset$ for all $m,m'\in M$).

We will write $[r,s]=\{r,r+1,\dots,s-1,s\}$ for the set of integers between $r$ and $s$ inclusive, and we will denote $[1,n]$ by $[n]$. The subsets of $[n]$ of size $t$ will be called \emph{layer} $t$ and we will denote them by
\[
    \binom{[n]}{t}=\{X\subseteq [n]\mid |X|=t\}.
\]
Similarly, let $\binom{[n]}{\geq t}=\{X\subseteq [n]\mid |X|\geq t\}$ be the subsets of size at least $t$ and $\binom{[n]}{\leq t}=\{X\subseteq [n]\mid |X|\leq t\}$ the subsets of size at most $t$.  For the set of all subsets of a set $X$, we use the notation $2^{X}$. For a set system $\F\subseteq 2^{[n]}$, we denote the collection of subsets of size $t$ in $\F$ by $\mathcal{F}_t$.

We will often consider the \emph{Hasse diagram} of $2^{[n]}$ where there is an edge from $X \subseteq [n]$ to $Y \subseteq [n]$ if $X \subseteq Y$ and $|Y| = |X|+1$.

A \emph{chain} $C\subseteq 2^{[n]}$ is a set system consisting of pairwise comparable elements, that is, $X\subseteq Y$ or $Y\subseteq X$ for all $X,Y\in C$.
An \emph{antichain} is a set system consisting of elements that are pairwise incomparable.

We say a chain $C$ in $2^{[n]}$ \emph{starts} in $R$ and \emph{ends} in $S$ if the smallest element of $C$ is in $R$ and the largest element of $C$ is in $S$. We say a chain $C_1 \subseteq \dotsb \subseteq C_m$  is \emph{skipless} if $|C_{i+1}| = |C_{i}| + 1$ for all $i \in [m-1]$ i.e. the chain does not `skip' over any layers.

A \emph{chain decomposition} of a set system $\F\subseteq 2^{[n]}$ is a collection of disjoint chains $C^1,\dots,C^m\subseteq \F$ such that $\F=\cup_{i=1}^m C^i$, that is, for each $X\in \F$, there is exactly one $i\in [m]$ such that $C^i$ contains the set $X$. The size of the chain decomposition is the number of chains $m$.

We assume the reader to be familiar with the following immediate consequence of Dilworth's theorem.
\begin{theorem}[Dilworth \cite{dilworth1950decomposition}]
    Let $n$ be an integer and $\F\subseteq 2^{[n]}$. The size of the largest antichain in $\F$ is equal to the minimal size of a chain decomposition of $\F$.
\end{theorem}

The symmetric chain decomposition described in  \cite{aigner,GK76} gives the following result.
\begin{lemma}
    \label{lem:symchain}
    There is a skipless chain decomposition of $2^{[n]}$ into $\binom{[n]}{\lfloor n/2\rfloor}$ chains. In particular, there is a matching of size $\binom{n}s$ from $\binom{[n]}{s}$ to $\binom{[n]}{r}$ whenever $s<r\leq \lceil n/2\rceil$ or $s>r\geq \lfloor n/2\rfloor$.
\end{lemma}

\subsection{Colex for shadows and matchings}
\label{sec:shadow_match}
In the \textit{colexicographic} or \textit{colex} order  on $\binom{[n]}t$, we have $A<B$ if $\max(A\triangle B) \in B$, where $\triangle$ denotes the symmetric difference $A\triangle B=(A\setminus B)\cup (B\setminus A)$.
Informally, sets with larger elements come later in the order. For $t=3$ the initial segment of size 8 in  colex is
given by
\[
    \{1, 2, 3\}, \{1, 2, 4\}, \{1, 3, 4\}, \{2, 3, 4\}, \{1, 2, 5\}, \{1, 3, 5\}, \{2, 3, 5\}, \{1, 4, 5\}.
\]
We write $\C(m,t)$ for the initial segment of colex on layer $t$ of size $m$.

For a family of sets $A \subseteq \binom{[n]}{t}$, the \emph{shadow} of $\A$ is given by
\[
    \partial\A = \{X\in \tbinom{[n]}{t-1}\mid X\subseteq Y \text{ for some }Y\in \A\}.
\]

The well-known Kruskal-Katona theorem below shows that the shadow of a family of subsets of size $t$ is minimised by taking the family to be an initial segment of colex, and we will prove an analogous result about matchings between a family and its shadow.

\begin{theorem}[Kruskal-Katona \cite{kruskal}]
    \label{thm:KKT}
    Let $1\leq t\leq n$ be integers.
    Let $\B\subseteq \binom{[n]}t$ and let $\C$ be the initial segment of colex on $\binom{[n]}t$ of size $|\B|$. Then $|\partial \B|\geq |\partial \C|$.
\end{theorem}

For $\B\subseteq \binom{[n]}t$, let $\nu(\B)$ denote the size of the maximum matching in the bipartite graph between $\B$ and $\partial \B$, where $X\in \B$ is adjacent to $Y\in \partial \B$ if $Y\subseteq X$.

\begin{lemma}
    \label{lem:colex_best}
    Let $1\leq t\leq n$ be integers.
    Let $\B\subseteq \binom{[n]}t$ and let $\C$ be the initial segment of colex on $\binom{[n]}t$ of size $|\B|$. Then $\nu(\B)\geq \nu(\C)$.
\end{lemma}
We could not find a reference for this result, so we will provide a proof for completeness.
Our proof relies on the Kruskal-Katona theorem and the following variant of Hall's theorem.
\begin{lemma}
    \label{lem:Hall}
    Let $G=(U,V,E)$ be a bipartite graph. The largest matching in $G$ between $U$ and $V$ has size $|U|-d$, where
    \[
        d = \max_{X\subseteq U} (|X|-|N(X)|).
    \]
\end{lemma}

\begin{proof}[Proof of Lemma \ref{lem:colex_best}]
    We prove the lemma by induction on $|\B|$. When $|\B|=1$, $\nu(\B)=1$ for all $\B\subseteq \binom{[n]}t$. If $\nu(\B)=|\partial \B|$, then the Kruskal-Katona theorem (Theorem \ref{thm:KKT}) gives
    \[
        \nu(\B)=|\partial\B|\geq |\partial \C|\geq \nu(\C).
    \]

    We now assume that $\nu(\B)<|\partial\B|$, and show that there is a $B\in \B$ for which $\nu(\B\setminus\{B\})<\nu(\B)$.
    The lemma then follows by induction. Indeed, let $B$ denote the element for which $\nu(\B\setminus\{B\})<\nu(\B)$. Let $C\in \C$ denote the last element of $\C$ in the colex order. Then $\C\setminus \{C\}$ is an initial segment of colex so the induction hypothesis shows that $\nu(\B\setminus \{B\})\geq \nu(\C\setminus\{C\})$. Hence
    \[
        \nu(\B)> \nu(\B\setminus \{B\})\geq \nu(\C\setminus\{C\})\geq \nu(\C)-1.
    \]
    Since all numbers are integers, we find $\nu(\B)\geq \nu(\C)$ as desired.

    It remains to show the claim that there is a $B\in \B$ for which $\nu(\B\setminus\{B\})<\nu(\B)$ when $\nu(\B) < |\partial\B|$.
    Consider the bipartite graph $G$ between $U=\partial \B$ and $V=\B$ (where $u\in U$ is adjacent to $v\in V$ if $u\subseteq v$).
    By \cref{lem:Hall}, $\nu(\B)= |\partial \B|-d$ where
    \[
        d= \max_{\X \subseteq \partial \B}|\X|- |N(\X)|.
    \]
    Pick $\X \subseteq \partial \B$ such that $|\X|-|N(\X)|=d$, and note that $d \geq 1$ by assumption. This means $\X$ is non-empty and there is some set $B \in N(\X)$. Consider the largest matching when we remove $B$ from $G$. In this graph $|\X| - |N(\X)|$ is $d+1$ and so applying  \cref{lem:Hall} shows that the largest matching between $\partial \B$ and $\B \setminus \{B\}$ is of size at most $|\partial \B| -  d - 1$.
    This proves that $\nu(\B\setminus\{B\})<\nu(\B)$, as desired.
\end{proof}

\subsection{Cascade notation}
\label{sec:cascade}
Let $m,r$ be integers.
For our upper bound construction, we need a result which gives the value of $\nu(\C(m,r))$.

There is a unique way of writing $m$ as
\[m = \binom{a_r}{r} + \binom{a_{r-1}}{r-1} + \dotsb + \binom{a_s}{s}\]
where $r\geq s \geq 1$, $a_r > \dotsb > a_s > 0$ and $a_i \geq i$ for all $i \in [s]$. The initial segment of colex $\C(m,r)$ of size $m$ on layer $r$ is the union of the set $\binom{[a_r]}{r}$, the set containing all elements of the form $A \cup \{a_r +1\}$ with $A \in \binom{[a_{r-1}]}{r - 1}$, the set containing all elements of the form $A \cup \{a_r +1, a_{r-1} + 1\}$ where $A \in \binom{[a_{r-2}]}{r-2}$, and so on until the sets containing all the elements of the form $A \cup \{a_r + 1, a_{r-1} + 1, \dots, a_{s+1} + 1\}$ where $A \in \binom{[a_{s}]}{s}$.

The expansion above is also called the \textit{$r$-cascade notation} of $m$  and may be built recursively as follows. Take $a_r$ to be the largest $j$ such that $\binom{j}{r} \leq m$, and set $m' = m - \binom{j}{r}$. If $m' = 0$, the recursion ends. Otherwise append the expansion for $m'$ and $r' = r - 1$.

This expansion can be used to compute the size of the shadow $|\partial\C(m,r)|$, but we are interested in using it to give the precise value of $\nu(\C(m,r))$ as follows.
\begin{lemma}
    \label{lem:colex-nu}
    Let $r\geq s \geq 1$ and $a_{r} > \dots > a_{s} > 0$ be such that \begin{equation}
        \label{eq:initial_exp}
        m = \binom{a_r}{r} + \binom{a_{r-1}}{r-1} + \dotsb + \binom{a_s}{s}.
    \end{equation}
    If $i\leq \lceil a_i/2\rceil$ for all $i\in [s,r]$, then $\nu(\C(m,r))=\sum_{i=s}^r\binom{a_i}{i-1}$. Otherwise, let $j\in [s,r]$ be maximal such that $j > \ceil{a_j/2}$. Then
    \[\nu(\C(m, r)) = \binom{a_r}{r-1} + \dotsb + \binom{a_{j+1}}{j} + \binom{a_j}{j} + \dotsb + \binom{a_s}{s}.\]
\end{lemma}

\begin{proof}
    The first claim follows from the fact that there is a matching of size $\binom{a_i}{i-1}$ between $\binom{[a_i]}{i}$ and $\partial \binom{[a_i]}{i}=\binom{[a_i]}{i-1}$ when $i\leq \lceil a_i/2\rceil$ (see Lemma \ref{lem:symchain}).

    We prove the second claim by induction on $r-j$, starting with $r - j = 0$. Note that $\C(m,r) \subseteq \binom{[a_{r} + 1]}{r}$, and that $\ceil{a_r/2} < r$ implies that $\floor{(a_r + 1)/2} \leq r - 1$. By Lemma \ref{lem:symchain}, there is a matching from $\binom{[a_{r} + 1]}{r}$ to $\binom{[a_{r} + 1]}{r - 1}$ of size $\binom{a_r + 1}{r}$ and this induces a matching of size $|\C|$ from $\C$ to $\partial \C$, as required.

    Now let $r - j \geq 1$, and suppose $M$ is a maximum matching from $\C$ to $\partial \C$. The sets in $\C$ which do not contain $a_r + 1$ are exactly $\binom{[a_r]}{r}$ and these can only be matched to sets in $\partial \C$ which do not contain $a_r + 1$, namely to sets in $\binom{[a_r]}{r-1}$. Since $r \leq \ceil{a_r/2}$, there is a matching $M'$ of size $\binom{a_r}{r-1}$ between $\binom{[a_r]}{r}$ and $\binom{[a_{r-1}]}{r-1}$ by \cref{lem:symchain}. We can assume that $M'\subseteq  M$ since elements in $\binom{[a_r]}{r}$ cannot be matched outside of $\binom{[a_r]}{r-1}$ by $M$.
    Let $\C'$ be obtained from $\C \setminus \binom{[a_r]}{r}$ by deleting the element $a_r + 1$ from every set. Then $\C'$ is an initial segment of colex of size $m-  \binom{a_r}{r}$ from layer $r-1$, and we apply the induction hypothesis to $\C'$ to get the result.
\end{proof}

\section{Generalisation of a result of Lehman-Ron}
\label{sec:gener}
We will prove Theorem \ref{thm:gener} from the following lemma using an inductive argument.
\begin{lemma}
    \label{lem:multi_complete_chains}
    Let $s \leq r \leq n$ be integers. Let $C^1,\dots,C^{m}$ be disjoint chains, such that for all $i\in [m-1]$, the chain $C^i$ starts in layer $s$ and ends in layer $r$. Suppose that $C^m$ starts in $A\in \binom{[n]}{\leq s}$ and ends in $B\in \binom{[n]}{r}$.
    Then there exist $m$ disjoint chains $D^1,\dots,D^{m}$ with the following three properties.
    \begin{enumerate}
        \item For $i\in [m-1]$, the chain $D^i$ starts in the $s$th layer, ends in the $r$th layer and is skipless.
        \item The chain $D^{m}$ starts at $A$ and intersects the $i$th layer for all $i\in [s+1,r]$.
        \item The chains $D^1,\dots,D^{m}$ cover the elements in $C^1,\dots,C^m$.
    \end{enumerate}
\end{lemma}

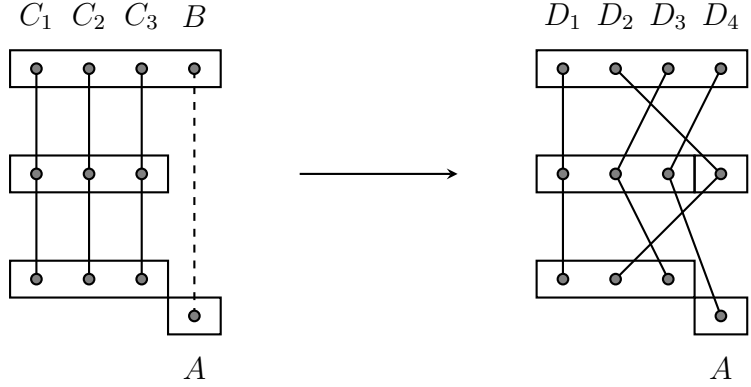
\begin{figure}
    \centering
    \tikzstyle{vertex}=[circle, draw, fill=black!50,
                        inner sep=0pt, minimum width=4pt]

\begin{tikzpicture}[thick,scale = 0.7]
      \foreach \pos in {0,1,2} {
      \draw (\pos,0) -- (\pos,4); 
      }
      \draw[dashed] (3,-0.7) -- (3,4);
      
      \draw (10,0) -- (10,4);
      \draw (11,0) -- (13,2) -- (11,4);
      \draw (12,0) -- (11,2) -- (12,4);
      \draw (13,-0.7) -- (12,2) -- (13,4);

      \foreach \shift in {0,10} {
      \foreach \pos in {0,1,2} {
      \node[vertex] at (\pos+\shift,0) {};
      };
      \draw[draw=black] (-0.5+\shift,-0.35) rectangle ++(3,0.7);
      
      \node[vertex] at (3+\shift,-0.7) {};
      \draw[draw=black] (2.5+\shift,-1.05) rectangle ++(1,0.7);
      
      \foreach \pos in {0,1,2} {
      \node[vertex] at (\pos+\shift,2) {};
      };
      \draw[draw=black] (-0.5+\shift,-0.35+2) rectangle ++(3,0.7);
      
      \foreach \pos in {0,1,2,3} {
      \node[vertex] at (\pos+\shift,4) {};
      };
      \draw[draw=black] (-0.5+\shift,-0.35+4) rectangle ++(4,0.7);
      };
      
      \node[vertex] at (3+10,2) {};
      \draw[draw=black] (-0.5+13,-0.35+2) rectangle ++(1,0.7);
      
     \draw [-stealth] (5,2) -- ++(3,0);
     
     \node at (0,5) {$C_1$};
     \node at (1,5) {$C_2$};
     \node at (2,5) {$C_3$};
     \node at (3,5) {$B$};
     \node at (3,-1.7) {$A$};
     
     \node at (10,5) {$D_1$};
     \node at (11,5) {$D_2$};
     \node at (12,5) {$D_3$};
     \node at (13,5) {$D_4$};
     \node at (13,-1.7) {$A$};
\end{tikzpicture}
    \caption{Representation of \cref{lem:multi_complete_chains} (case $r = s+2$ and $m=3$).}
    \label{fig:chain}
\end{figure}
Note that the lemma allows the element $A$ to appear on a lower layer than the others (illustrated in Figure \ref{fig:chain}) and that it may be impossible to add an element on layer $s$ to the chain $D^m$.

The overall structure of the proof of Lemma \ref{lem:multi_complete_chains} is very similar to that of Lehman-Ron \cite{lehman2001disjoint}.
We first consider the special case in which $s=r-2$. As in the proof of Lehman-Ron \cite{lehman2001disjoint}, the first step is to show that there are at least $m$ elements in the $(r-1)$th layer that could be elements of the chains $D_1,\ldots,D_m$.
\begin{lemma}
    \label{lem:size_Q}
    Let $a,r,n$ be integers satisfying  $a \leq r-2 \leq n-2$, let $\R \subseteq {[n] \choose r}$ be of size $m$, let $\S \subseteq {[n] \choose r-2}$ be of size $m-1$ and let $A \in {[n] \choose a}$. Suppose that there exists a bijection $f: \R \to \S \cup \{A\}$ with $f(X) \subseteq X$ for all $X \in \R$. \\
    Let $\Q$ denote the set of $Q\in {[n] \choose r-1}$ with {$S \subseteq Q \subseteq R$} for some {$(S,R) \in (\S \cup \{A\}) \times \R$}. Then $|\Q| \geq m$.
\end{lemma}
\begin{proof}
    We prove the claim by contradiction. Consider a counterexample to the claim for which $m$ is minimal. If $m=1$, then we are given elements $A \subseteq R$ with $|A|\leq r-2$ and $|R|=r$. Then there exists at least one element $Q \in \Q$ such that $A \subseteq Q \subseteq R$: simply remove one of the elements in $R \setminus A$ from $R$ to obtain $Q$. We therefore assume $m\geq 2$.

    We consider the Hasse diagram $H = (V,E)$ of $2^{[n]}$. Note that $\Q$ can be seen as the set of all elements of cardinality $r-1$ lying on a path between an element of $\S \cup \{A\}$ and an element of $\R$.

    We consider the `restriction' $H'=(V',E')$ which is obtained by taking the subgraph of $H$ on vertex set $V'=\R\cup \S \cup \Q \cup \{A\}$, removing all arcs containing $A$ and then adding an arc from $A$ to $Q$ for all $Q \in \Q$ with $A \subseteq Q$. We denote by $N^{+}(X)$ (resp. $N^{-}(X)$) the set of vertices $Y$ with an arc $X \to Y$ (resp. with an arc $Y \to X$) in $H'$, and define $d^{+}(X)=|N^{+}(Y)|$ and $d^{-}(X)=|N^{-}(X)|$.
    We first prove the following three claims.
    \begin{claim}
        \label{claim:in_neighbour}
        For every $R \in \R$ and every $Q \in N^{-}(R)$, we have $d^{-}(R) \geq d^{-}(Q).$
    \end{claim}
    \begin{proof}
        In order to prove the claim, for any $R$ in $\R$, and any $Q$ in $N^{-}(R)$, we exhibit an injective function $\pi : N^{-}(Q) \to N^{-}(R)$.

        We denote by $j$ the unique element of the set $R\setminus Q$. For $S \in \S \cap N^{-}(Q)$, we denote by $i$ the unique element in $Q \setminus S$ and set $\pi(S)=R \setminus \{i\}=S \cup \{j\}$. Note that $R \setminus \{i\} \in \Q$ as $S,R\setminus \{i\},R$ is a path in $H'$. By doing so, we specified a unique $\pi(S)\in N^{-}(R)$ for all $S\in N^{-}(Q)$ except for possibly $A$ if $A\in N^{-}(Q)$. However, there is one element in $N^{-}(R)$ that we have not yet used: the element $R\setminus \{j\}\in N^{-}(R)$ and we may map this element to $A$ to finish the definition of our injection $\pi$ if needed.
    \end{proof}

    \begin{claim}
        \label{claim:out_neighbour}
        For every $Q \in \Q$ and every $S \in N^-(Q)$, we have $d^+(Q) < d^+(S)$.
    \end{claim}

    \begin{proof}
        The proof of this claim is similar to the proof of the previous claim. Let $Q \in \Q$ and $S \in N^-(Q)$. Once again we exhibit an injective function $\pi' : N^+(Q) \rightarrow N^+(S)$.
        We define $\pi'(R) = S \cup (R \setminus Q)$ for $R \in N^{+}(Q)$. Note that $Q$ itself is never an image of $\pi'$ thus a strict inequality holds.
    \end{proof}

    \begin{claim}
        \label{claim:sum_neighbour}
        $$
            \sum_{R \in \R} d^{-}(R) \geq \sum_{Q \in \Q} d^{-}(Q) \quad \text{ and } \quad
            \sum_{Q \in \Q} d^{+}(Q)  < \sum_{S \in \S\cup \{A\}} d^{+}(S).
        $$
    \end{claim}

    \begin{proof}
        We start by showing the first inequality.
        For $c \in \mathbb{N}$, let us define $\R_c = \{R \in \R \mid d^{-}(R) = c\}$ and distinguish two cases. Suppose first that there exists a $c \in \mathbb{N}$ such that $\R_c = \R$. Then we have $\sum_{R \in \R} d^{-}(R) = cm$. By Claim \ref{claim:in_neighbour}, $\forall Q \in \Q$, $d^{-}(Q) \leq c$ and therefore $\sum_{Q \in \Q} d^{-}(Q) \leq c(m-1) \leq \sum_{R \in \R} d^{-}(R)$.

        Otherwise, $\R_c \neq \R$ for every choice of $c$. In this case, we define, for any integer $d < \max_{R\in \R} d^-(R)$, {$\R_{\leq d} = \cup_{c \leq d} ~ \R_c$} and remark that $\R_{\leq d} \neq \R$.
        Since we chose $(\R,\S \cup \{A\})$ to be minimal, Lemma \ref{lem:size_Q} holds for the pair $(\R_{\leq d},f(\R_{\leq d}))$. In particular, we can find a set $\Q_{\leq d}$ of size exactly $|\R_{\leq d}|$ such that $\Q_{\leq d} \subseteq \Q$ and every element in $\Q_{\leq d}$ lies on a path between an element of $\R_{\leq d}$ and an element of $f(\R_{\leq d})$. By definition, each $Q\in \Q_{\leq d}$ is in the in-neighbourhood of some $R\in \R_{\leq d}$, and therefore $d^-(Q)\leq d$ by Claim \ref{claim:in_neighbour}. We conclude that for any $d < \max_{R\in \R} d^-(R) $ there exist $|\Q_{\leq d}| = |\R_{\leq d}|$ vertices in $\Q$ of in degree at most $d$.

        If we denote by $d_0<d_1,\ldots<d_k$ the in-degree sequence of $\R$, then the result of the last paragraph induces an injective function $\pi'':\R_{\leq d_{k-1}}\to \Q$ as follows: we map $\R_{d_0}$ to $\Q_{\leq d_0}$, then map $\R_{d_1}$ to $\Q_{\leq d_1} \setminus \Q_{\leq d_0}$ and continue to map $\R_{d_i}$ to $\Q_{\leq d_i} \setminus \Q_{\leq d_{i-1}}$ for all $i\in [2,k-1]$. We argued in the previous paragraph that such injections exist.
        By construction, $\forall R \in \R_{\leq d_{k-1}}, ~ d^-(\pi''(R)) \leq d^-(R)$.

        All vertices in $\Q$ are in the in-neighbourhood of some element of $\R$ and therefore $d^-(Q)\leq d_k$ for all $Q\in \Q$ by Claim \ref{claim:in_neighbour}. Since by assumption $|\Q|<|\R|$, this proves  $\sum_{R \in \R} d^{-}(R) \geq \sum_{Q \in \Q} d^{-}(Q)$ since we can associate each term in the second sum to an element that is at least as large in the first sum (and all terms are non-negative).

        The proof of the inequality $\sum_{Q \in \Q} d^{+}(Q)  < \sum_{S \in \S\cup \{A\}} d^{+}(S)$ is analogous, but now the strict inequality follows from the strict inequality in Claim \ref{claim:out_neighbour} instead of the weak inequality of Claim \ref{claim:in_neighbour}.
    \end{proof}

    We are now fully equipped to conclude the proof of Lemma \ref{claim:out_neighbour}. By double counting, $\sum_{Q \in \Q} d^{-}(Q)  = \sum_{S \in \S \cup \{A\}} d^{+}(S)$ and $\sum_{R \in \R} d^{-}(R)  = \sum_{Q \in \Q} d^{+}(Q)$. Using Claim \ref{claim:sum_neighbour} we deduce the following contradiction,
    $$
        \sum_{Q \in \Q} d^{-}(Q) = \sum_{S \in \S\cup \{A\}} d^{+}(S) > \sum_{Q \in \Q} d^{+}(Q) = \sum_{R \in \R} d^{-}(R) \geq \sum_{Q \in \Q} d^{-}(Q).
    $$
    This proves the lemma.
\end{proof}
Using Lemma \ref{lem:size_Q} we can now prove the following special case of Lemma \ref{lem:multi_complete_chains}, which we will use to push through an inductive argument.
\begin{lemma}
    \label{lem:complete_chains}
    Let $3\leq r\leq n$ be integers. Let $C^1,\dots,C^{m-1}$ be skipless disjoint chains between the $(r-2)$th and the $r$th layers. Let $B\in \binom{[n]}{r}$ and let $A$ be a subset of $B$ of size at most $r-2$, such that $A,B \notin \cup_{i=1}^{m-1} C^i$.

    Then there exist $m$ disjoint chains $D^1,\dots,D^{m}$ with the following three properties.
    \begin{itemize}
        \item For $i\in [m-1]$, the chain $D^i$ starts in the $(r-2)$th layer, ends in the $r$th layer and is skipless.
        \item The chain $D^{m}$ starts in $A$ and intersects both the $(r-1)$th and the $r$th layer.
        \item The chains $D^1,\dots,D^{m}$ cover the elements in $C^1,\dots,C^{m-1}$ and $A,B$.
    \end{itemize}
\end{lemma}
\begin{proof}
    We prove the claim by induction on $m$. The case $m=1$ is immediate.

    We let $\R,\T,\S$ denote the restriction of the chains to layers $r,r-1,r-2$ respectively, and add $A$ to $\S$ and $B$ to $\R$. That is,
    \begin{align}
        \R & = \left( \bigcup_i C^i \cap {[n] \choose r}\right) \cup \{B\}, \nonumber    \\
        \T & = \bigcup_i C^i \cap {[n] \choose r-1}, \nonumber                           \\
        \S & = \left( \bigcup_i C^i \cap {[n] \choose r-2} \right) \cup \{A\}. \nonumber
    \end{align}
    Let $\Q$ denote the set of all elements $Q \in {[n] \choose r-1}$ such that there exists $(R,S) \in \R \times \S$ satisfying $R \subseteq Q \subseteq S$. We define a bijection $f:\R\to \S$ with $f(B)=A$ and $f(X)\subseteq X$ for all $X\in \R$ using the given chains.
    Lemma \ref{lem:size_Q} shows that $|\Q|\geq m$. Since $|\T|=m-1$, $\T$ is a strict subset of $\Q$.

    We consider the poset as a directed graph $H'$ via an adjusted Hasse diagram as before: the vertex set consists of $V=\R\cup \Q \cup \S$, and $X \rightarrow Y$ is an arc in $E$ if and only if $X \subseteq Y$ and either $|Y| = |X|+1$ or $X=A$ and $Y\in \Q$. Finding the desired chains $D^1,\dots,D^m$, is equivalent to finding $m$ vertex-disjoint paths between $\R$ and $\S$ in the induced subgraph $H_{Q}=H'[\R \cup \T \cup \S \cup \{Q\}]$ for some $Q\in \Q$. By Menger's theorem \cite{Menger1927}, $m$ vertex-disjoint paths exist if and only if there is no $(\R,\S)$-cut of size $m-1$, that is, there is no subset $\C\subseteq V$ with $|\C|=m-1$ such that any path between any pair $(R,S) \in \R \times \S$ contains a vertex of $\C$.

    Since $|\T|<|\Q|$, there is an element $Q_0\in \Q\setminus \T$.
    By the discussion above, we may assume that an $(\R,\S)$-cut $\C$ of size $m-1$ exists in $H_{Q_0}$.
    We first show that $\C \not \subseteq \Q$. Indeed, for any $Q \in \Q$ there exists a pair $(R,S) \in \R \times \S$ such that $S\rightarrow Q\rightarrow R$ is a path in $H'$. When $\C\subseteq \Q$, all such paths in $H_{Q_0}$ are cut off only when $\C$ contains all elements of $\T\cup\{Q_0\}$; but $|\C| = m-1 < m=|\T \cup \{Q_0\}|$. So $\C$ must contain at least one element which is not in $\Q$.

    We partition the size of the cut in three parts
    $$
        m_1 = |\R \cap \C|, \quad m_2 = |\Q \cap \C|, \quad m_3 = |\S \cap \C|.
    $$
    Consider the chains whose endpoints have not been touched by the cut. That is, let $\R^*\subseteq \R$ consist of the $R\in \R$ for which $R,f(R)\not\in \C$, and let $\S^*=f(\R^*)$. Then $\Q \cap \C$ is an $(\R^*,\S^*)$-cut. Moreover,
    \[
        m_2=|\Q \cap \C|=(m-1)-m_1-m_3<m-m_1-m_3\leq |\R^*|.
    \]
    Let $\T^*\subseteq \T$ consist of the elements that lie on some chain $C^i$ between $\S^*$ and $\R^*$. Since $\Q \cap \C$ is an $(\R^*,\S^*)$-cut of $H_{Q_0}$, it must in particular contain all elements of $\T^*$. Since $m_2 <  |\R^*|$, this means that $(A,B)\in (\S^*\times \R^*)$. Moreover, we may apply the induction hypothesis since $|\R^*| < |\R|$ (because $m_1+m_3>0$). This gives us $|\R^*|$ chains which cover all elements in $\T^*$ and all intersect layer $r-1$, so in particular we obtain some element $Q_1\in \Q\setminus \T^*$ such that there are $|\R^*|> m_2$ vertex-disjoint $\S^*-\R^*$ paths in $H^*=H'[\R^*\cup \T^* \cup \{Q_0,Q_1\}\cup \S^*]$.
    We distinguish two cases.
    \begin{itemize}
        \item Suppose that $Q_1 \notin \T$. In this case we have obtained our desired chain decomposition. Indeed, we keep the chains between $\S\setminus \S^*$ and $\R\setminus\R^*$ as they are and since $\T^*\cup \{Q_1\}$ is disjoint from those chains, we may use the $|\R^*|$ chains between $\R^*$ and $\S^*$ that we obtained by induction in order to define the remaining chains.
        \item Suppose that $Q_1 \in \T$. In that case, $H^*$ is an induced subgraph of $H_{Q_0}$. This gives a contradiction:  $H^*$ has $|\R^*|>m_2$ vertex disjoint paths between $\R^*$ and $\S^*$, whereas $\Q\cap \C$ gives an $(\R^*,\S^*)$-cut of size $m_2$ in $H_{Q_0}$. \qedhere
    \end{itemize}
\end{proof}
From this, we will deduce the case of general $s$.
\begin{proof}[Proof of Lemma \ref{lem:multi_complete_chains}]
    We prove the lemma by induction on $m$. The case $m=1$ is immediate. Suppose the claim has been shown for all $m'<m$.

    Let $C^1,\dots,C^m$ be the given chain decomposition, where $C^m$ starts in $A\in \binom{[n]}{\leq s}$ and ends in $B\in \binom{[n]}r$, and the first $m-1$ chains are between layers $s$ and $r$.
    Let $t \in [s+1,r]$. We say the chains $D^1,\dots,D^m$ are $t$-good if the first $m-1$ chains are skipless and between layers $s$ and $r$, $D^m$ is between $A$ and $B$ and intersects layers $t,\dots, r$, and $\cup_{i=1}^{m} C^i\subseteq \cup_{i=1}^{m} D^i$.

    We first argue that there exists an $r$-good decomposition.
    Indeed, applying the induction hypothesis to the first $m'=m-1$ chains, we can find chains $D^1,\dots,D^{m-1}$ between layers $s$ and $r$ that are skipless and such that $\cup_{i=1}^{m-1} C^i\subseteq \cup_{i=1}^{m-1} D^i$. By removing the elements from $C^{m}$ that also appear in some $D^i$, we have obtained an $r$-good decomposition for $C^1,\dots,C^m$.

    Let $t\leq r$ be minimal for which a $t$-good decomposition $D^1,\dots,D^m$ exists.
    Suppose towards a contradiction that $t>s+1$.
    Let $B'$ be the element of $D^m$ in layer $t$. Since $t>s+1$, we find $t-2\geq s$ and so the chains $D^1,\dots,D^{m-1}$ all intersect layer $t-2$. We can apply \cref{lem:complete_chains} on the chains $D^1,\dots,D^{m-1}$ restricted to layers $s'=t-2$ and $r'=t$, and elements $A$ and $B'$. This produces  a set $\mathcal{C}_1$ of chains. Let $\mathcal{C}_0$ and $\mathcal{C}_2$ be the restrictions of $D^1,\dots,D^{m}$ to layers $s,\dots,t-2$ and to layers $t,\dots,r$ respectively. Then each chain of $\mathcal{C}_1$ shares a vertex with exactly one chain of $\mathcal{C}_0$ and exactly one chain of $\mathcal{C}_2$. Hence, there is only one way to merge these chains in a chain decomposition $E^1,\dots,E^m$. This chain decomposition is $(t-1)$-good, contradicting the minimality of $t$. Therefore, there exists an $(s+1)$-good decomposition $D^1,\dots,D^m$, as claimed by the lemma.
\end{proof}

We will obtain \cref{thm:gener} as a corollary of the following lemma. The lemma is stated in the way that we wish to apply it in the proof of \cref{thm:main}.
\begin{lemma}
    \label{lem:main}
    Let $\F\subsetneq 2^{[n]}$ be $k$-antichain saturated. Then $\F$ has a chain decomposition into $k-1$ skipless chains.
\end{lemma}
\begin{proof}
    Suppose, towards a contradiction, that $\F$ has no chain decomposition $C^1,\dots,C^{k-1}$ for which the first $i+1$ chains are skipless, but it does have one for which the first $i$ are skipless. By \cref{lem:multi_complete_chains} applied to a single chain, we can always rearrange the chains such that $C^1$ is skipless. This means we have $1\leq i<k-1$.

    Amongst the decompositions for which the first $i$ chains are skipless, we choose a decomposition $C^1,\dots,C^{k-1}$ which minimises the `number of layers the $(i+1)$th chain skips'. That is, the decomposition which minimises
    \[
        \max_{X\in C^{i+1}}|X|-\min_{Y\in C^{i+1}} |Y|+1-|C^{i+1}|.
    \]
    By assumption, we can find $A\subseteq B$ consecutive in $C^{i+1}$ with $|B|>|A|+1$ such that $C^{i+1}$ is skipless between $B$ and its maximal element. After renumbering, we can assume that for some $j\in [0,i]$, the chains $C^1,\dots,C^j$ have elements present on layers $|B|-2,|B|-1$ and $|B|$, whereas $C^{j+1},\dots,C^{i}$ miss an element either on layer $|B|-2$ or on layer $|B|$. (Here we use that $C^1,\dots,C^i$ are skipless.) In particular, if $C^a$ where $a \in [j+1, i]$ has an element on layer $|B|-1$, then it is its minimal or maximal element, and so we can move it to another chain without creating any skips in the chain $C^a$.

    We apply Lemma \ref{lem:multi_complete_chains} to the chains $C^1,\dots,C^j$ restricted to layers $|B|-2,|B|-1,|B|$,
    and $A\subseteq B$ to obtain disjoint chains $D^1,\dots,D^{j+1}$ with the following properties:
    \begin{itemize}
        \item $\cup_{a=1}^{j}C^a\cup \{A,B\}\subseteq \cup_{a=1}^{j+1}D^a$;
        \item $D^1,\dots,D^{j}$ are skipless, start in layer $|B|-2$ and end in layer $|B|$;
        \item $D^{j+1}$ contains $A$ and elements on layers $|B|-1$ and $|B|$.
    \end{itemize}
    Since the chains $C^1,\dots,C^j$ have an element on layer $|B|-1$,  there is a unique $X\in \cup_{a=1}^{j+1}D^a$ with $|X|=|B|-1$ such that $X\not\in \cup_{a=1}^jC^a$.

    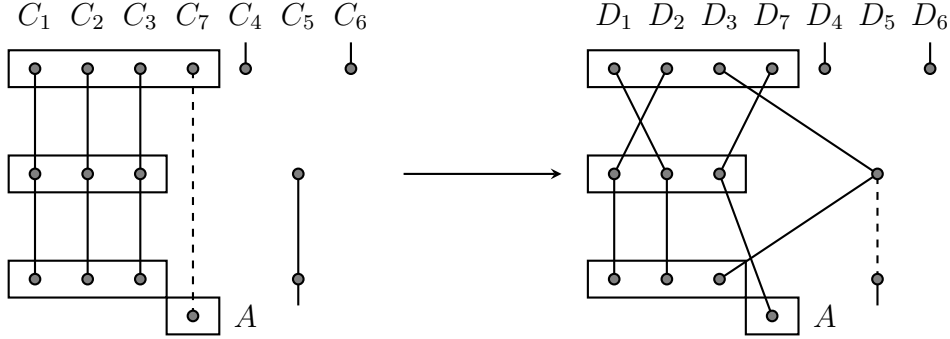
\begin{figure}
        \centering
        \tikzstyle{vertex}=[circle, draw, fill=black!50,
                        inner sep=0pt, minimum width=4pt]

\begin{tikzpicture}[thick,scale = 0.7]
      \foreach \pos in {0,1,2} {
      \draw (\pos,0) -- (\pos,4); 
      }
      \draw[dashed] (3,-0.7) -- (3,4);
      
      \foreach \shift in {11} {
     \draw (\shift,0) -- (\shift,2) -- (\shift+1,4);
     \draw (\shift+1,0) -- (\shift+1,2) -- (\shift,4);
     \draw (\shift+2,0) -- (\shift+5,2) -- (\shift+2,4);
     \draw (\shift+3,-0.7) -- (\shift+2,2) -- (\shift+3,4);
     \node at (\shift,5) {$D_1$};
     \node at (\shift+1,5) {$D_2$};
     \node at (\shift+2,5) {$D_3$};
     \node at (\shift+3,5) {$D_7$};
     \node at (\shift+4,5) {$D_4$};
     \node at (\shift+5,5) {$D_5$};
     \node at (\shift+6,5) {$D_6$};
     \node at (\shift+4,-.7) {$A$};
     \draw (5+\shift,-0.5) -- (5+\shift,0);
     \draw[dashed] (5+\shift,0) -- (5+\shift,2);
     \draw (4+\shift,4) -- (4+\shift,4.5);
     \draw (6+\shift,4) -- (6+\shift,4.5);
      }
      
     \draw (5,-0.5) -- (5,2);
     \draw (4,4) -- (4,4.5);
     \draw (6,4) -- (6,4.5);

      \foreach \shift in {0,11} {
      \foreach \pos in {0,1,2} {
      \node[vertex] at (\pos+\shift,0) {};
      };
      \draw[draw=black] (-0.5+\shift,-0.35) rectangle ++(3,0.7);
      
      \node[vertex] at (3+\shift,-0.7) {};
      \draw[draw=black] (2.5+\shift,-1.05) rectangle ++(1,0.7);
      
      \foreach \pos in {0,1,2,5} {
      \node[vertex] at (\pos+\shift,2) {};
      };
      \draw[draw=black] (-0.5+\shift,-0.35+2) rectangle ++(3,0.7);
      
      \foreach \pos in {0,1,2,3,4,6} {
      \node[vertex] at (\pos+\shift,4) {};
      };
      \draw[draw=black] (-0.5+\shift,-0.35+4) rectangle ++(4,0.7);
      \node[vertex] at (\shift+5,0) {};
      };
     \draw [-stealth] (7,2) -- ++(3,0);
     
     \node at (0,5) {$C_1$};
     \node at (1,5) {$C_2$};
     \node at (2,5) {$C_3$};
     \node at (3,5) {$C_7$};
     \node at (4,5) {$C_4$};
     \node at (5,5) {$C_5$};
     \node at (6,5) {$C_6$};
     
     \node at (4,-.7) {$A$};
\end{tikzpicture}
        \caption{An example of a possible rearrangement as done in the proof of \cref{lem:main} (for $j = 3$ and $i=6$). The sets $A$ and $B$ are part of the chain $C_7$.}
        \label{fig:rearr}
    \end{figure}
    The chains $D^1,\dots,D^{j+1}$ define a matching $M$ between layer $|B|$ and layer $|B|-1$ of size $j+1$. We will use this to reroute the chains into a `better' chain decomposition and arrive at a contradiction. A possible configuration is depicted in Figure \ref{fig:chain}.
    We define the chain decomposition $E^1,\dots,E^{k-1}$ as follows.

    For $a\in [j]$, let $b\in [j]$ be such that $D^b$ contains the unique element in $C^a$ of size $|B|-2$. Let $a'\in [j]\cup \{i+1\}$ be the index such that $D^b$ contains the unique element in $C^{a'}$ of size $|B|$. We set
    \[
        E^{a}= \left[C^{a}\cap \binom{[n]}{\leq |B|-2}\right] \cup D^b \cup \left[ C^{a'}\cap \binom{[n]}{\geq |B|}\right].
    \]
    Note that our assumption that $C^{1}, \dots, C^{i+1}$ are skipless from layer $|B|$ upwards means the chain $E^a$ must be skipless as well.

    For $a\in [j+1,i]$, we let $E^a=C^a\setminus\{X\}$. Either we kept the chain the same, or we removed the minimal or maximal element, so these chains are also skipless.
    For $a\in [i+2,n]$,  we also set $E^a=C^a\setminus\{X\}$.

    For $a=i+1$, let $C^{a'}$ be the unique  chain which contains the element of $D^{j+1}$ of size $|B|$. We set \[
        E^{i+1}= \left[C^{i+1}\cap \binom{[n]}{\leq |A|}\right] \cup D^{j+1} \cup \left[C^{a'}\cap \binom{[n]}{\geq |B|}\right].
    \]
    The chains $E^1,\dots,E^{k-1}$ form a chain decomposition of $\F\cup \{X\}$ (which must equal $\F$ in this case because $\F$ is $k$-antichain saturated).
    The chains $E^1,\dots,E^i$ are skipless and the chain $E^{i+1}$ skips one fewer layers than the chain $C^{i+1}$, contradicting the optimality of $C^1,\dots,C^{k-1}$.
\end{proof}
We recall the statement of \cref{thm:gener}: if $\F$ admits a chain decomposition into $m$ chains, then it can be covered by $m$ skipless chains.
\begin{proof}[Proof of \cref{thm:gener}]
    By assumption, $\F$ does not contain an antichain of size $m+1$.
    Let $\F'$ be obtained from $\F$ by greedily adding sets until the set system has become $(m+1)$-antichain saturated. If $\F'=2^{[n]}$, then we find a skipless chain decomposition for $\F'$ by \cref{lem:symchain}. Otherwise, we can find a chain decomposition for $\F'$ into $m+1-1=m$ skipless chains by \cref{lem:main}. These chains cover $\F$ as desired.
\end{proof}

\section{Lower bounds for antichain saturation numbers}
\label{sec:proof}
In this section, we prove the lower bound of Theorem \ref{thm:main}. We first recall the set-up.
Given a natural number $k$, let $\ell$ be the smallest integer $j$ such that $\binom{j}{\floor{j/2}} \geq k - 1$. We may assume that $n\geq \ell$.
Let $\C(m,t)$ denote the initial segment of layer $t$ of size $m$ under the colexicographic order.
Let $c_{t} =k-1$ for all $t\in [\floor{\ell/2},\floor{n/2}]$. For $0\leq t<\floor{\ell/2}$, we define $c_t = \nu\left(\C(c_{t+1},t+1)\right)$.

The lower bound of Theorem \ref{thm:main} follows directly from the lemma below, since the desired lower bound for the upper layers follows by symmetry.
\begin{lemma}
    \label{lem:main_lower}
    For any $k$-antichain saturated set system $\F\subsetneq 2^{[n]}$,
    $|\F_t| \geq c_t$ for any $t\leq \lfloor n/2\rfloor$.
\end{lemma}
\begin{proof}
    Suppose that $\F\subsetneq 2^{[n]}$ is $k$-antichain saturated.
    By Lemma \ref{lem:main}, there is a skipless chain decomposition $C^1,\dots,C^{k-1}$ for $\F$. Let $\F_t=\F\cap \binom{[n]}t$.  We define $\mathcal{D}(\mathcal{F}_t)$ as the sets $A\in \F_{t-1}$ for which the chain $C^i$ that contains $A$ also contains an element of $\F_t$.
    The following claim is key to our proof.
    \begin{claim}
        \label{cl:down_nu}
        For all $t\in [n]$, $|\mathcal{D}(\mathcal{F}_t)| = \nu(\mathcal{F}_t)$.
    \end{claim}
    \begin{proof}
        By definition, there is a matching from $\F_t$ to $\D(\F_t) \subseteq \partial \F_t$ of size $|\D(\F_t)|$, and hence,  $\nu(\F_t) \geq |\D(\F_t)|$. We now focus on the opposite inequality.

        Suppose, towards a contradiction, that there is a $t$ for which $|\mathcal{D}(\mathcal{F}_t)| < \nu(\mathcal{F}_t)$. Let $M$ be a matching between $\F_t$ and $\partial \F_t$ of size $\nu(\mathcal{F}_t)$, and let $M'$ be the matching between $\F_t$ and $\C(\F_t)$ corresponding to the inclusions in the chains (i.e. $X$ is matched to $Y$ if $X$ and $Y$ are in the same chain). Consider the multigraph where the vertices are $\binom{[n]}{t} \cup \binom{[n]}{t-1}$ and the edge set is $M \cup M'$. The non-empty components of this graph are paths and even cycles which alternate between edges $M$ and $M'$ (with no multiedges), and multiedges which have one edge from $M$ and one edge from $M'$. Since $|M| > |M'|$ there must be some component $P$ which is a path that starts and ends with edges from $M$. We will reroute some of the chains so that they use the edges from $M$ instead of the edges from $M'$, increasing the size of $\D(\F_t)$.

        If a chain $C^a$ is not incident with an edge in this path, let $D^a = C^a$ (i.e. the chain is unchanged). One end of $P$ must be in layer $t$ and one end in layer $t-1$, and we order the edges starting from the end in layer $t$. If $e \in M$ is not the last edge in the path, then it connects a set $X \in C^a$ of size $t$ to a set $Y \in C^b$ of size $t-1$, and we replace $C^a$ by \[D^a= \left(C^{a}\cap \binom{[n]}{\geq t}\right)\cup\left( C^b \cap \binom{[n]}{\leq t-1}\right).\]

        If $e \in M$ is the last edge in the path, there are two cases. The edge may connect a set $X \in C^a$ of size $t$ to a set $Y$ of size $t-1$ which is not in any other chain, in which case we set $D^a = \left(C^{a}\cap \binom{[n]}{\geq t}\right) \cup \{Y\}$. Then $D^1, \dots, D^{k-1}$ gives a decomposition of $\mathcal{F} \cup\{Y\}$ into $k-1$ chains and this contradicts the assumption that $\mathcal{F}$ is $k$-antichain saturated. The other case is where the edge connects a set $X \in C^a$ of size $t$ to a set $Y \in C^b$ of size $t-1$. Since there is no edge in $M'$ incident to $Y$, it must be the largest set in $C^b$. In this case, we set $D^a = \left(C^{a}\cap \binom{[n]}{\geq t}\right) \cup C^b$. The $k-2$ chains $D^1,\dots,D^{b-1},D^{b+1},\dots,D^{k-1}$ now cover all the elements of $\F$ and we may still define the chain $D^b$ freely. We can choose any set $A$ which is not in $\mathcal{F}$ and set $D^{b}=\{A\}$. Then $D^1, \dots D^{k-1}$ is a chain decomposition of $\F \cup \{A\}$ into $k-1$ chains, a contradiction.
    \end{proof}
    \cref{lem:symchain} shows that, for all $t > \floor{n/2}$ there is a matching between $\F_t$ and $\partial \F_t$ of size $|\F_t|$, which implies $\nu(\F_t)= |\F_t|$. Using the claim above, every chain with a set in layer $t$ must have a set in layer $t-1$ for all $t>\floor{n/2}$.
    The set system $\overline{\mathcal{F}}=\{[n]\setminus F:F\in \F\}$ is also $k$-antichain saturated. Applying \cref{cl:down_nu} to $\overline{\F}$, we find that every chain with a set of size $s < \ceil{n/2}$ must have a set of size $s+1$ as well. Putting these together gives the following claim.
    \begin{claim}
        \label{cl:middle_layer}
        For all $i\in [k-1]$, $C^i$ contains an element from layer $\lfloor n/2\rfloor$.
    \end{claim}
    An immediate consequence of \cref{cl:down_nu} is that $|\F_{t-1}|\geq \nu(\F_t)$. Together with Lemma \ref{lem:colex_best}, this shows
    \begin{equation}
        \label{eq:push_down}
        |\F_{t-1}|\geq \nu(\F_t)\geq \nu(\C),
    \end{equation}
    where $\C$ is an initial segment of colex on $\binom{[n]}t$ of size $|\F_t|$. We already have $|\F_{\floor{n/2}}| = k-1$ and we want this for $\F_t$ down to $t = \floor{\ell/2}$. From (\ref{eq:push_down}),  we can push this downwards at least when $\nu(\C) = |\C|$, and the following claim shows that this holds for all $t >  \floor{\ell/2}$.

    \begin{claim}
        \label{cl:nu_for_colex}
        For $t> \lfloor \ell/2\rfloor $, an initial segment of colex $\C$ on layer $t$ of size at most $k-1$ has $\nu(\C) = |\C|$, and so $|\F_{t-1}| \geq |\F_{t}|$.
    \end{claim}

    \begin{proof}
        Let $\ell^*$ be the largest element in any set in $\mathcal{C}$ i.e. $\ell^* = \max (\bigcup_{A \in \mathcal{C}} A)$. If $t > \floor{\ell^*/2}$, then applying Lemma \ref{lem:symchain} to $[\ell^*]$ shows that there is a matching from $|\C|$ to layer $t - 1$ of $[\ell^*]$ of size $|\C|$, and we are done. Suppose instead that $t \leq \floor{\ell^*/2}$. Since $\C$ is an initial segment of colex, it must contain all subsets of $[\ell^* - 1]$ of size $t$ as well as a set containing $\ell^*$, but this means $\C$ contains too many sets. Indeed,
        \[1 + \binom{\ell^* - 1}{t} \geq 1 + \binom{2t - 1}{t} \geq 1 + \binom{\ell}{\floor{\ell/2}} \geq  k.\qedhere\]
    \end{proof}
    Combined with \cref{cl:middle_layer}, we find that layers $\lfloor \ell/2\rfloor$ up to $\lfloor n/2\rfloor$ all contain $k-1$ elements of $\F$.

    For $t< \lfloor \ell/2\rfloor$, if $|\F_{t+1}|\geq c_{t+1}$ then (\ref{eq:push_down}) shows that
    \[
        |\F_t|\geq \nu(\C(|\F_{t+1}|,t+1))\geq \nu(\C(c_{t+1},t+1))= c_{t}.
    \]
    Hence by induction, we find that
    \[
        \sum_{t=0}^{\lfloor n/2\rfloor}|\F_t|\geq (\lfloor n/2\rfloor-\lfloor \ell/2\rfloor)(k-1) + \sum_{t=0}^{\lfloor\ell/2 \rfloor} c_t. \qedhere
    \]
\end{proof}
The lower bound in \cref{cor:nicecase} follows from the above using the simple observation that $\nu\left( \binom{[m]}{r}\right) = \binom{m}{r-1}$ provided $r \leq \ceil{m/2}$.

When $k-1=\binom{\ell}{\lfloor \ell/2\rfloor}$, we believe that all minimal $k$-antichain saturated set systems have a similar shape: layer $\lfloor \ell/2\rfloor$ is the lowest layer with $k-1$ elements and induces an isomorphic copy of colex, layer $n-\lfloor \ell/2\rfloor$ is the highest layer with $k-1$ elements and contains the complements of an isomorphic copy of an initial segment of colex, and the elements in between these two layers can be covered by $k-1$ skipless chains.

\section{Upper bound constructions}
\label{sec:upperbound}
We first describe the known upper bound construction in the case where $k -1$ is a central binomial coefficient. Combining this construction with the lower bound above gives \cref{cor:nicecase}. We then give the upper bound construction for \cref{thm:main} which works for all values of $k$, but requires a slightly larger value of $n$. Ignoring the minor changes we make to allow for smaller $n$, the first construction is a special case of the second construction.

\subsection{The upper bound construction of \cref{cor:nicecase}}
\label{sec:upper_cor}
Let $\ell,k,n$ be integers such that $\binom{\ell}{\lfloor \ell/2\rfloor}=k-1$ and
$n\geq \ell+1$. For the sake of completeness, we give a precise description of the construction from \cite{ferrara2017saturation}, which shows
\begin{equation}
    \label{eq:cor_sat_upper}
    \sat(n,k)\leq 2 \sum_{j=0}^{\lfloor \ell/2\rfloor} \binom{\ell}j + (k-1)(n-1-2\lfloor \ell/2\rfloor).
\end{equation}
We define a set system $\F \subseteq 2^{[n]}$ that is $k$-antichain saturated.

For $t \leq \lfloor \ell/2 \rfloor$, the sets of size $t$ in $\mathcal{F}$ are exactly the subsets of $[\ell]$ of size $t$, and for $t\geq n-\lfloor \ell/2\rfloor$, we add to $\F$ all subsets $X\subseteq [n]$ of size $t$ for which $[n]\setminus X$ is a subset of $[\ell]$. There are $k-1$ sets $\F$ of size $\floor{\ell/2}$ and $n - \floor{\ell/2}$, and we will join these up using \cref{thm:gener}.

For $\ell$ odd, we first fix a matching $M$ between $\binom{[\ell]}{\lfloor \ell/2\rfloor }$ and $\binom{[\ell]}{\lceil \ell/2\rceil }$, which exists by Lemma \ref{lem:symchain}. When $\ell$ is even, we let $M$ be the identity. We denote by $M(X)$ the element matched to $X$ by $M$. Let $f: \mathcal{F}_{\floor{\ell/2}} \to \mathcal{F}_{n - \floor{\ell/2}}$ be given by
\[
    f(X) =  M(X)\cup [\ell+1,n],
\]
and note that $X\subseteq f(X)$ for all $X \in \mathcal{F}_{\floor{\ell/2}}$.
To complete the family $\F$, we take any set of $k-1$ disjoint skipless chains between $\mathcal{F}_{\floor{\ell/2}}$ and $\mathcal{F}_{n - \floor{\ell/2}}$, which we know exist by Theorem \ref{thm:gener} (or the result of Lehman and Ron).

To see that $\F$ has no antichain of size $k$, we note that it allows a decomposition into $k-1$ chains. Indeed, we may extend the previously obtained $k-1$ chains between layers $\lfloor \ell/2\rfloor$ and $n-\lfloor \ell/2\rfloor$, using any chain decomposition of $[\ell]$ restricted to the lowest $\lfloor \ell/2\rfloor$ layers. We can similarly extend the chains to the layers $n-\lfloor \ell/2\rfloor + 1,\dots,n$.

To see that $\F$ is saturated, note that we clearly cannot add any subset of size $t\in [\lfloor \ell/2\rfloor,n-\lfloor \ell/2\rfloor]$ since $\F$ already contains $k-1$ subsets of size $t$. For $t<\lfloor \ell/2\rfloor$, any subset of size $t$ that is not yet in $\F$ must contain some element $i>\ell$ and is therefore incomparable to the $k-1$ elements of $\F\cap \binom{[n]}{\lfloor \ell/2\rfloor}$. A similar argument holds for $t>n-\lfloor \ell/2\rfloor$.

By counting the number of sets in each layer, we find that
\[|\mathcal{F}| =  2 \sum_{j=0}^{\lfloor \ell/2\rfloor} \binom{\ell}j + (k-1)(n-1-2\lfloor \ell/2\rfloor),\]
as required.

\subsection{Upper bound for Theorem \ref{thm:main}}
In this subsection, we prove the following lemma, which is the upper bound of \cref{thm:main}. Recall that for a given $k \geq 1$, we let $\ell$ be the smallest integer $j$ such that $\binom{j}{\floor{j/2}} \geq k - 1$, and we recursively define $c_{0},c_1, \dots, c_{\floor{\ell/2}}$ as follows.
Let $c_{\floor{\ell/2}} = k-1$ and, for $0\leq t<\floor{\ell/2}$, let $c_t = \nu\left(\C(c_{t+1},t+1)\right)$. We will show that there is a $k$-antichain saturated set system $\F$ where $\F_t$ and $\F_{n-t}$ contain $c_t$ sets for all $t \leq \floor{\ell/2}$, which gives the following result.
\begin{lemma}
    \label{lem:upper_main}
    Using the notation above, \[
        \sat(n,k) \leq  2 \sum_{t=0}^{\floor{\ell/2}} c_t+(k-1)(n - 1 - 2\floor{\ell/2})
    \]
    provided $n \geq 2\ell  + 1$.
\end{lemma}

Lemma \ref{lem:colex_best} shows that  $\nu(\B)$ is minimised by taking $\B$ to be an initial segment of colex.
We wish to construct a set system $\F$, such that $\nu(\F_t)=\nu(\C(|\F_t|,t)$ for all $t\leq \lfloor \ell/2\rfloor$, yet $\F$ can be covered by $k-1$ chains.

Suppose that each set in $\C(m,r)$ is in a chain and consider how many continue to the layer below. \cref{lem:colex-nu} shows that the only `savings' $m-\nu(\C(m,r))$ come from the initial sequence where $j \leq \ceil{a_j/2}$, and that we need to continue the chains for the remaining terms to the layer below. This gives us some freedom to change those terms in the $r$-cascade notation of $m$ (see (\ref{eq:initial_exp})), and we will modify the terms at the end so that they are part of the expansion of a later layer. We will do this using a different way of writing of $m$ as a sum of binomial coefficients that we now introduce.

Given $m, r \geq 1$ such that $m \geq \binom{2r-1}{r}$.
Let the \emph{$r$-expansion} of $m$ be
\[
    m = \binom{a_{r_1}}{r_1} + \dotsb + \binom{a_{r_s}}{r_s}
\]
recursively formed as follows. Let $r_1 = r$ and define $a_{r_1}$ as the maximum $j$ such that $\binom{j}{r_1} \leq m$. Note that $a_{r_1}\geq 2r_1-1$.
Set $m' = m - \binom{a_{r_1}}{r_1}$. If $m' = 0$, we are done. Otherwise, let $r'$ be the maximum $j \leq r -1$ such that $\binom{2j-1}{j} \leq m'$ and form the $r$-expansion of $m$ by appending to $\binom{a_{r_1}}{r_1}$ the $r'$-expansion of $m'$. It is easy to see that this is well-defined and must terminate.

As an example, let us consider the $5$-expansion of $m = 1011$. Since $\binom{12}{5} < 1011 < \binom{13}{5}$, we take $a_{r_1} = 12$ (and $r_1 = 5$).  This means $m' = 219$, and the largest integer $j \leq 4$ such that $\binom{2j -1}{j} \geq m'$ is $ j = 4$ (we also have $\binom{9}{5} \leq m'$, but this is not allowed). We therefore append the $4$-expansion of 219. Calculating this recursively in the same manner, we see $a_{r_2} = 10$ (and $r_2 = 4$), which leaves a remainder of $9$. Since $\binom{4}{2} \leq 9 < \binom{5}{3}$, we append the $2$-expansion of $9$, which is $\binom{4}{2} + \binom{3}{1}$. This gives the $5$ expansion of 1011 as
\[1011 = \binom{12}{5} + \binom{10}{4} + \binom{4}{2} + \binom{3}{1}.\]

The following lemma gives some properties of the $r$-expansion of an integer $m$.
\begin{lemma}
    \label{lem:r-expansion}
    Let $m,r \geq 1$ be such that $m \geq \binom{2r-1}{r}$. Let the  $r$-expansion of $m$ be \[m = \binom{a_{r_1}}{r_1} + \dotsb + \binom{a_{r_s}}{r_s}.\]
    Then the following statements hold:
    \begin{enumerate}
        \item $r = r_1 > \dotsb > r_s \geq 1$;
        \item $a_{r_1} > \dots > a_{r_s} \geq 1$;
        \item for all $i \in [s]$, we have $r_i \leq \ceil{a_{r_i}/2}$.
    \end{enumerate}
\end{lemma}
\begin{proof}
    We prove this by induction on $s$. There is nothing to prove for the base case $s = 1$. Suppose that $s \geq 2$. Using $m'$ and $r'$ as in the definition of the $r$-expansion and applying the induction hypothesis to the $r'$-expansion of $m'$ (which has $s - 1$ terms), the following must hold:
    \begin{itemize}
        \item $r' = r_2 > \dotsb > r_s \geq 1$;
        \item $a_{r_2} > \dots > a_{r_s} \geq 1$;
        \item for all $i \in [2,s]$, we have $r_i \leq \ceil{a_{r_i}/2}$.
    \end{itemize}
    By definition we have $r = r_1 > r_2$ and it follows from $m \geq \binom{2r-1}{r}$, that $r_1 \leq \ceil{a_{r_1}/2}$. Hence, we only need to check that $a_{r_1} > a_{r_2}$.

    Suppose first that $r' = r - 1$. If $a_{r_2} \geq a_{r_1}$, we have
    \[ m \geq \binom{a_{r_1}}{r_1} + \binom{a_{r_1}}{r_1 -1} = \binom{a_{r_1} + 1}{r_1}\]
    which contradicts the definition of $a_{r_1}$. If $r' \leq r - 2$, then $a_{r_2} \geq a_{r_1} \geq 2r' +3$, and so
    \[\binom{2r' + 3}{r'} \geq \binom{2r' + 1}{r'} = \binom{2(r'+1) - 1}{r' + 1}.\] This contradicts our choice of $r'$.
\end{proof}
By the lemma, $r_1\geq r_2-1\geq r_3-2\geq  \dots \geq r_s -(s-1)$.
Using the observation that $\binom{m+1}{r} = \sum_{i = 0}^r \binom{m - i}{r-i}$, we obtain the following simple lemma.

\begin{lemma}
    \label{lem:lower_expansion}
    Suppose that $m, r \geq 1$ satisfy $m \geq \binom{2r-1}{r}$. Let the $r$-expansion of $m$ be \[
        m = \binom{a_{r_1}}{r_1} + \dotsb + \binom{a_{r_s}}{r_s}.\]
    Fix $t \in[r]$ and let $i$ be the unique integer such that $r_i + (i-1) \geq t > r_{i+1} + i$ (where we take $r_{s+1} = - s -  1$). Suppose that $i \leq s - 1$ and define $m'$ by
    \[
        m' = \binom{a_{r_1}}{t} + \dots + \binom{a_{r_{i}}}{t - (i-1)} + \binom{a_{r_{i+1}}}{r_{i+1}} + \dotsb + \binom{a_{r_s}}{r_s}.\]
    Then the expression above gives the $t$-expansion of $m'$.
\end{lemma}
\begin{proof}
    Suppose that the $t$-expansion of $m'$ is not as claimed, but that instead the $t$-expansion of $m'$ is
    \[m' = \binom{a'_{r'_1}}{r'_1} + \dotsb + \binom{a'_{r'_{p}}}{r'_{p}}.\]
    Let $j$ be the first point at which this expansion differs from the claimed expansion. If $j \geq  i + 1$ then
    \[\binom{a_{r_{i+1}}}{r_{i+1}} + \dotsb + \binom{a_{r_s}}{r_s} = \binom{a'_{r'_{i+1}}}{r'_{i+1}} + \dotsb + \binom{a'_{r'_p}}{r'_p},\]
    but looking at the definition of the $r$-expansion, both of these come from exactly the same recursion.

    Instead, we must have $j \leq i$. We first prove $r_{j}'=t-(j-1)$. Note that $r_1'=t$ so $j>1$ and $r'_{j-1} =t - (j - 2)$. This implies $r'_{j}\leq t-(j-1)$.
    From the definition of $i$, we find that $r_j+(j-1)\geq r_i+(i-1)\geq t$ and so $r_j\geq t-(j-1)$.
    Let $m'' = m' - \binom{a_{r_1}}{t} - \dotsb - \binom{a_{r_{j-1}}}{t - (j-2)}$.    Then $m'' \geq  \binom{a_{r_j}}{t-j+1}\geq \binom{2(t - j) + 1}{t -j + 1}$ and so $r'_j \geq t - (j-1)$.

    Since $r_j'=t-(j-1)$, it must be the case that $a_{r_j}\neq a_{r_j'}'$. However,
    \begin{align*}
        m'' & = \binom{a_{r_{j}}}{t - (j - 1)} + \dotsb + \binom{a_{r_{i}}}{t - (i-1)} + \binom{a_{r_{i+1}}}{r_{i+1}} + \dotsb + \binom{a_{r_s}}{r_s} \\
            & \leq \binom{a_{r_{j}}}{t - (j - 1)} + \binom{a_{r_{j}} - 1}{t - j} + \dotsb + \binom{a_{r_{j}} - (t-j)}{1}                              \\
            & < 1 + \binom{a_{r_{j}}}{t - (j - 1)} + \binom{a_{r_{j}} - 1}{t - j} + \dotsb + \binom{a_{r_{j}} - (t-j)}{1}                             \\
            & = \binom{a_{r_j} + 1}{t-(j-1)}.
    \end{align*}
    That is
    \[\binom{a_{r_{j}}}{t - (j - 1)} \leq m '' < \binom{a_{r_j} + 1}{t-(j-1)},\]
    and $a_{r_j'}'=a_{r_j}$ by definition.
\end{proof}

We are now ready to prove \cref{lem:upper_main}.

\begin{proof}[Proof of Lemma \ref{lem:upper_main}]
    Let $\ell$ be the smallest integer $j$ such that $\binom{j}{\floor{j/2}} \geq k - 1$. We have
    \[k - 1 \geq \binom{\ell - 1}{\floor{(\ell - 1)/2}} = \binom{\ell -1}{\ceil{(\ell -1)/2}} = \binom{\ell -1}{\floor{\ell /2}} \geq \binom{2 \floor{\ell/2} - 1}{\floor{\ell /2}}.\]
    Let the $\floor{\ell/2}$-expansion of $k-1$ be
    \[ k - 1 = \binom{a_{r_1}}{r_1} + \dotsb + \binom{a_{r_s}}{r_s}\]
    where $\floor{\ell/2} = r_1 > \dotsb > r_s \geq 1$, $a_{r_1} > \dots > a_{r_s} > 0$ and $r_i \leq \ceil{a_{r_i}/2}$ for all $i \in [s]$. These facts are guaranteed by Lemma \ref{lem:r-expansion}. By our assumption that $\binom{\ell}{\floor{\ell/2}} \geq k -1$, we have $a_{r_1} \leq \ell -1$.

    We now define our construction by processing each of the terms in this expansion.
    Initialise $\I$ as an empty set of chains. For each $i \in [s]$, let $\A_i$ be the set system consisting of sets of the form \[A = X \cup \{a_{r_1} + 1, a_{r_2} + 1, \dots, a_{r_{i-1}} + 1\}\] where $X$ is a subset of $[a_{r_i}]$ of size at most $r_i$. Note that the largest element in any of these sets is either $a_{r_1}$ or $a_{r_1} + 1$, and hence all sets are contained in $[\ell]$.

    Since $r_i \leq \ceil{a_{r_i}/2}$, we can cover $\A_i$ with $\binom{a_{r_i}}{r_i}$ disjoint chains, and we add these chains to our collection of chains $\I$. Indeed, we may start with the chains from a symmetric chain decomposition of $2^{[a_{r_i}]}$ and add the elements $a_{r_1} + 1, a_{r_2} + 1, \dots, a_{r_{i-1}} + 1$ to every set. Discard any sets which are not in $\A_i$ and remove any empty chains.

    Define $f: 2^{[n]} \to 2^{[n]}$ by $f(A) = \{i \in [n] : n + 1 - i \not \in A\}$. Form a set of chains $\I'$ be replacing each chain $C \in \I$ by $\{A \cup f(A): A \in C\}$. Since we have assumed that $n \geq 2\ell  + 1$, we have that $A\subseteq f(A)$ for any set $A\subseteq [\ell]$, and these are indeed chains.
    The chains in $\I'$ are also disjoint and we can apply Theorem \ref{thm:gener} to find disjoint chains $D^{1}, \dots, D^{k-1}$ which cover the same sets and are skipless. We take $\F$ to be the union of $D^1, \dots, D^{k-1}$. See \cref{fig:constr} for a depiction of a set system constructed as above.
    \begin{figure}
        \centering
        \tikzstyle{vertex}=[circle, draw, fill=black!50,
                        inner sep=0pt, minimum width=4pt]

\begin{tikzpicture}[thick,scale = 0.45]
      \foreach \pos in {0,1,...,4} {
      \node[vertex] at (\pos,0) {};
      \node[vertex] at (\pos,4) {};
      };
      \draw[draw=black] (-0.5,-0.35) rectangle ++(5,0.7);
      \draw (-0.5,-0.35) -- ++(0,-6) -- ++(5,6);
      \draw[draw=black] (-0.5,3.65) rectangle ++(5,0.7);
      \draw (-0.5,3.65+0.7) -- ++(0,6) -- ++(5,-6);
      
      \foreach \pos in {0,1,2} {
      \node[vertex] at (\pos+6,-2) {};
      \node[vertex] at (\pos+6,6) {};
      };
      \draw[draw=black] (5.5,-2.35) rectangle ++(3,0.7);
      \draw (5.5,-2.35) -- ++(0,-3) -- ++(3,3);
      \draw[draw=black] (5.5,5.65) rectangle ++(3,0.7);
      \draw (5.5,5.65+0.7) -- ++(0,3) -- ++(3,-3);

      \foreach \pos in {0,1} {
      \node[vertex] at (\pos+10,-4) {};
      \node[vertex] at (\pos+10,8) {};
      }
      \draw[draw=black] (9.5,-4.35) rectangle ++(2,0.7);
      \draw (9.5,-4.35) -- ++(0,-0.5) -- ++(2,0.5);
      \draw[draw=black] (9.5,7.65) rectangle ++(2,0.7);
      \draw (9.5,7.65+0.7) -- ++(0,0.5) -- ++(2,-0.5);
      
      \draw (-0.75,-7) -- ++(0,-0.5) -- ++(5.5,0) -- ++(0,0.5);
      \draw (5,-7) -- ++(0,-0.5) -- ++(3.75,0) -- ++(0,0.5);
      \draw (9,-7) -- ++(0,-0.5) -- ++(2.75,0) -- ++(0,0.5);
      
      \node at (2.25,-8) {$\mathcal{A}_1$};
      \node at (6.90,-8) {$\mathcal{A}_2$};
      \node at (10.5,-8) {$\mathcal{A}_3$};
      
      \node  at (13.5,0) (r1) {$r_1$};
      \draw [dashed] (4.5,0) -- (r1);
     
      \node at (14.1,-2) (r2) {$r_2 + 1$};
      \draw [dashed] (8.5,-2) -- (r2);
      
      \node at (14.1,-4) (r3) {$r_3 + 2$};
      \draw [dashed] (11.5,-4) -- (r3);
      
      \node at (14.4,4) (nr1) {$n - r_1$};
      \draw [dashed] (4.5,4) -- (nr1);
      
      \node at (15,6) (nr2) {$n - r_2 - 1$};
      \draw [dashed] (8.5,6) -- (nr2);
      
      \node at (15,8) (nr3) {$n - r_3 - 2$};
      \draw [dashed] (11.5,8) -- (nr3);
      
      \draw[draw=black] (-0.5,-0.35+0.7) rectangle ++(5,4);
      \draw[draw=black] (5.5,-2.35+0.7) rectangle ++(3,8);
      \draw[draw=black] (9.5,-4.35+0.7) rectangle ++(2,12);
\end{tikzpicture}
        \caption{The shape of our upper bound construction is depicted.}
        \label{fig:constr}
    \end{figure}
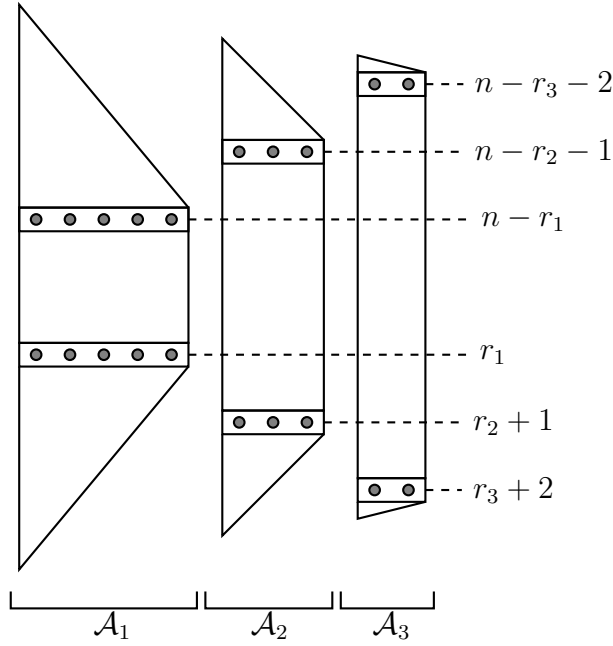

    The set system $\F$ is the union of $k-1$ chains, so cannot contain an antichain of size $k$. We now argue that adding any set $X$ to $\F$ creates an antichain of size $k$. Let $|X| = t$. If  $\floor{\ell/2} \leq t \leq n - \floor{\ell/2}$, then adding $X$ creates an antichain of size $k$ as $\F$ already contains $k - 1$ sets in layer $t$. We can assume that $1\leq t < \floor{\ell/2}$, else we can consider adding the set $f(X) = \{i \in [n] : n + 1 - i \not \in X\}$ instead. Set $r_{s+1} = -s -1$. The sequence $\floor{\ell/2} = r_1, r_2 + 1, \dots, r_{s+1} + s = -1$ is non-increasing so there is a unique $i$ such that $r_i + (i-1) \geq t > r_{i+1} + i$.
    Let $\D_j = \left\{ A \cup \{a_{r_1} + 1, \dots, a_{r_{j-1}} + 1\} : A \in \binom{[a_{r_j}]}{r_j}\right\}$ and let $\D$ be the set of all sets in $\mathcal{F}_t$ which contain $a_{r_1} + 1, \dots, a_{r_{i}} + 1$.

    We claim that the following subset of $\F \cup \{X\}$ is an antichain of size $k$:
    \[
        \left(\bigcup_{j = 1}^i \D_j\right) \cup \{X\} \cup \mathcal{D}.
    \]
    By definition, $|\D_j|=\binom{a_j}{r_j}$ for $j\in [i]$ and $|\D|=\sum_{j=i+1}^s\binom{a_j}{r_j}$ so indeed the collection has size $k$.
    We note that for $A,B\in \left(\bigcup_{j = 1}^i \D_j\right)  \cup \mathcal{D}$ with $|A|<|B|=r_j+(j-1)$, the element $a_{r_j}+1\in A\setminus B$.
    It is immediate that $\D\cup \{X\}$ forms an antichain since all sets have the same size.
    For $j\in [i]$ and $A\in \D_j$, the set $A$ cannot be a subset of $X$ as $|A|=r_j+(j-1)\geq t$.
    We next show that $X$ is not a subset of $A$.
    Suppose towards a contradiction that $X\subseteq A\in \D_j$. Then $a_{r_j}+1\not \in X$.
    Let $x\in [j]$ be the smallest integer such that  $a_{r_x}+1\not\in X$. Then, since $a_{r_y}< a_{r_x}$ for all $y> x$, we find that $X\setminus\{a_{r_1}+1,\dots,a_{r_{x-1}}+1\}\subseteq [a_{r_x}]$ and so
    $X\in \A_x\subseteq \F$, contradicting our choice $X\not\in \F$. This shows the set system is indeed an antichain.

    We have shown that $\F$ is $k$-antichain saturated.
    It remains to check that $\mathcal{F}$ has the claimed size. By Claim \ref{cl:down_nu}, we know that $|\F_{t-1}|= \nu(\F_t)$ for $t \leq  \floor{\ell/2}$ and $|\F_{\floor{\ell/2}}|=k-1=c_{\floor{\ell/2}}$. Hence, $|\F_t| \geq c_t$ for all $t \leq \floor{\ell/2}$. Let $t < \floor{\ell/2}$ be maximal such that $|\F_{t}| > c_{t}$, and define $i$ to be the unique integer such that $r_i + (i-1) \geq t > r_{i+1} + i$ (where we again take $r_{s+1} = -s - 1$). Suppose first that $i \leq s - 1$, so that
    \[
        |\F_t| = \binom{a_{r_1}}{t} + \dots + \binom{a_{r_{i}}}{t - (i-1)} + \binom{a_{r_{i+1}}}{r_{i+1}} + \dotsb + \binom{a_{r_s}}{r_s}
    \]
    is the $t$-expansion of $|\F_t|$ (using \cref{lem:lower_expansion}).
    We may also write $|\F_t|$ in its $t$-cascade notation (see (\ref{eq:initial_exp})) as
    \[ |\F_t| = \binom{b_t}{t} + \binom{b_{t-1}}{t-1} + \dotsb + \binom{b_{s'}}{s'}\]
    where $b_t > \dotsb > b_{s'}$, $s' \geq 1$ and $b_j \geq j$ for all $j \in [s', t]$.
    Note that by the way that both $a_{r_j}$ and $b_j$ are defined, we must have $b_t = a_{r_1}$, $b_{t-1} = a_{r_2}$, and so on until $b_{t-i + 1} = a_{r_i}$. That is,
    \[ |\F_t| = \binom{a_{r_1}}{t} + \dotsb + \binom{a_{r_{i}}}{t-(i-1)} + \binom{b_{t-i}}{t-i}+ \dotsb + \binom{b_{s'}}{s'}.\]
    If $t - i > \ceil{b_{t-i}/2}$, then Lemma \ref{lem:colex-nu} shows that $c_{t-1}$ is given by
    \[c_{t-1} = \binom{a_{r_1}}{t - 1} + \dotsb + \binom{a_{r_i}}{t - i} + \binom{b_{t -i}}{t - i} + \dotsb + \binom{b_{s'}}{s'},\]
    but this is exactly $|\F_{t-1}| = \nu(\F_t)$.
    If $t - i \leq \ceil{b_{t-i}/2}$, then we would take $r_{i+1} = t - i$ in the $t$-expansion of $|\F_t|$, but this contradicts the value of $i$.

    If $i = s$, then $\F_t$ is actually an initial segment of colex. All the elements of $\F_t$ come from the $\A_j$ and no sets are added to layer $t$ when we apply \cref{thm:gener}. If $s' = \min\{t + 1,s\}$, then there are no sets from $\A_j$ when $j > s'$ and when $j \leq s'$, the sets are of the form
    \[X \cup \{a_{r_1} + 1, a_{r_2} + 1, \dots, a_{r_{j-1}} + 1\}\] where $X$ is a subset of $[a_{r_j}]$ of size exactly $t - (j-1)$. This is exactly the initial segment of colex of size $|\F_t|$, and we have $c_{t-1} = |\F_{t-1}|$.
\end{proof}

\section{Open problems}
\label{sec:concl}
\cref{thm:main} gives the exact value of $\sat(n,k)$ for most values of $n$ and $k$, but it leaves a range $\ell \leq n \leq 2\ell$ for which the exact value is not known. For these values of $n$, the upper bound construction works separately on the lower half of the Boolean lattice and on the upper half, but it is not clear that we can join these constructions up.
With a little more care, we believe we can reduce this gap by proving the upper bound for slightly smaller values of $n$, but we do not know how to reduce this all the way down to $\ell$.
\begin{problem}
Is the lower bound of Theorem \ref{thm:main} correct for all $\ell \leq n\leq 2\ell$?
\end{problem}
Several other interesting questions concerning induced poset saturation number remain open, amongst which the question of which asymptotical behaviours are possible. Let $\sat(n,\mathcal{P})$ denote the smallest size of a set system $\mathcal{F}\subseteq 2^{[n]}$ which is (induced) $\mathcal{P}$-saturated. A recent result of \cite{Freschi2022} shows that either $\sat(n,P)=O(1)$ or $\sat(n,\mathcal{P})=2\sqrt{n-2}$. This result has been conjectured to be tight for the diamond \cite{ivandiamond} but it has also been conjectured by \cite{KESZEGH2021105497} that the lower bound can be improved to $\sat(n,\mathcal{P})\geq n+1$.

A natural question that one could ask is whether \cref{thm:gener} extends to other posets. The result of Lehman and Ron has already been generalised to a vast class of posets including geometric lattices by Logan and Shahriari \cite{logan2004new}, and our generalisation may also extend. Note that, following the proof of \cref{thm:gener}, it would suffice to prove an analogue of \cref{lem:size_Q}. An extension of Theorem \ref{thm:gener} to other posets might also make it possible to determine the asympotics of their antichain saturation numbers.

\bibliographystyle{alphaurl}
\bibliography{biblio}

\newcommand{\etalchar}[1]{$^{#1}$}
\begin{thebibliography}{KLM{\etalchar{+}}21}

\bibitem[AFT22]{Freschi2022}
Maryam~Sharifzadeh Andrea~Freschi, Simón~Piga and Andrew Treglown.
\newblock The induced saturation problem for posets.
\newblock {\em arXiv preprint arXiv:2207.03974}, 2022.

\bibitem[Aig73]{aigner}
Martin Aigner.
\newblock {Lexicographic matching in Boolean algebras}.
\newblock {\em Journal of Combinatorial Theory, Series B}, 14:187--194, 1973.

\bibitem[BS96]{Bajnok96}
Béla Bajnok and Shahriar Shahriari.
\newblock Long symmetric chains in the boolean lattice.
\newblock {\em Journal of Combinatorial Theory, Series A}, 75(1):44--54, 1996.
\newblock URL: \url{https://doi.org/10.1006/jcta.1996.0062}.

\bibitem[DHL19]{Duffus19}
Dwight Duffus, David Howard, and Imre Leader.
\newblock The width of downsets.
\newblock {\em European Journal of Combinatorics}, 79:46--59, 2019.
\newblock URL: \url{https://doi.org/10.1016/j.ejc.2018.11.005}.

\bibitem[Dil50]{dilworth1950decomposition}
Robert~P. Dilworth.
\newblock A decomposition theorem for partially ordered sets.
\newblock {\em Annals of Mathematics}, pages 161--166, 1950.

\bibitem[FKK{\etalchar{+}}17]{ferrara2017saturation}
Michael Ferrara, Bill Kay, Lucas Kramer, Ryan~R Martin, Benjamin Reiniger,
  Heather~C Smith, and Eric Sullivan.
\newblock {The saturation number of induced subposets of the Boolean lattice}.
\newblock {\em Discrete Mathematics}, 340(10):2479--2487, 2017.

\bibitem[GK76]{GK76}
Curtis Greene and Daniel~J. Kleitman.
\newblock {Strong versions of Sperner’s theorem}.
\newblock {\em Journal of Combinatorial Theory, Series A}, 20(1):80--88, 1976.

\bibitem[GKL{\etalchar{+}}13]{gerbner2013saturating}
D{\'a}niel Gerbner, Bal{\'a}zs Keszegh, Nathan Lemons, Cory Palmer,
  D{\"o}m{\"o}t{\"o}r P{\'a}lv{\"o}lgyi, and Bal{\'a}zs Patk{\'o}s.
\newblock {Saturating Sperner families}.
\newblock {\em Graphs and Combinatorics}, 29(5):1355--1364, 2013.

\bibitem[Gri82]{Griggs82}
Jerrold~R. Griggs.
\newblock {Collections of subsets with the Sperner property}.
\newblock {\em Transactions of the American Mathematical Society},
  269(2):575--591, 1982.

\bibitem[{\DJ}I22]{djankovic2022saturation}
Irina {\DJ}ankovi{\'c} and Maria-Romina Ivan.
\newblock Saturation for small antichains.
\newblock {\em arXiv preprint arXiv:2205.07392}, 2022.

\bibitem[Iva20]{ivanbutterfly}
Maria-Romina Ivan.
\newblock Saturation for the butterfly poset.
\newblock {\em Mathematika}, 66:806--817, 2020.

\bibitem[Iva22]{ivandiamond}
Maria-Romina Ivan.
\newblock Minimal diamond-saturated families.
\newblock {\em Contemporary Mathematics}, 3(2), 2022.

\bibitem[Kle75]{Kleitman75}
D.~J. Kleitman.
\newblock On an extremal property of antichains in partial orders. the lym
  property and some of its implications and applications.
\newblock In M.~Hall and J.~H. van Lint, editors, {\em Combinatorics}, pages
  277--290. Springer Netherlands, 1975.

\bibitem[KLM{\etalchar{+}}21]{KESZEGH2021105497}
Balázs Keszegh, Nathan Lemons, Ryan~R. Martin, Dömötör Pálvölgyi, and
  Balázs Patkós.
\newblock Induced and non-induced poset saturation problems.
\newblock {\em Journal of Combinatorial Theory, Series A}, 184:105497, 2021.
\newblock \href {https://doi.org/https://doi.org/10.1016/j.jcta.2021.105497}
  {\path{doi:https://doi.org/10.1016/j.jcta.2021.105497}}.

\bibitem[Kru63]{kruskal}
Joseph~B Kruskal.
\newblock The number of simplices in a complex.
\newblock {\em Mathematical optimization techniques}, 10:251--278, 1963.

\bibitem[Log02]{Logan02}
Mark~J. Logan.
\newblock Sperner theory in a difference of boolean lattices.
\newblock {\em Discrete Mathematics}, 257:501--512, 2002.

\bibitem[LR01]{lehman2001disjoint}
Eric Lehman and Dana Ron.
\newblock On disjoint chains of subsets.
\newblock {\em Journal of Combinatorial Theory, Series A}, 94(2):399--404,
  2001.

\bibitem[LS04]{logan2004new}
Mark~J Logan and Shahriar Shahriari.
\newblock A new matching property for posets and existence of disjoint chains.
\newblock {\em Journal of Combinatorial Theory, Series A}, 108(1):77--87, 2004.

\bibitem[Men27]{Menger1927}
Karl Menger.
\newblock Zur allgemeinen kurventheorie.
\newblock {\em Fundamenta Mathematicae}, 10(1):96--115, 1927.

\bibitem[MNS14]{scott}
Natasha Morrison, Jonathan~A. Noel, and Alex Scott.
\newblock {On saturated $k$-Sperner systems}.
\newblock {\em The Electronic Journal of Combinatorics}, 21(3):P3.22, 2014.

\bibitem[MSW20]{martin2019improved}
Ryan~R Martin, Heather~C Smith, and Shanise Walker.
\newblock Improved bounds for induced poset saturation.
\newblock {\em The Electronic Journal of Combinatorics}, 27(2):P2.31, 2020.

\bibitem[Sak79]{Saks79}
Michael Saks.
\newblock A short proof of the existence of $k$-saturated partitions of
  partially ordered sets.
\newblock {\em Advances in Mathematics}, 33:207--2011, 1979.

\bibitem[Spe28]{Sperner28}
Emanuel Sperner.
\newblock {Ein Satz über Untermengen einer endlichen Menge}.
\newblock {\em Mathematische Zeitschrift}, 27(1):544--548, 1928.
\newblock URL: \url{https://doi.org/10.1007/BF01171114}.

\bibitem[STW22]{Sudakov22}
Benny Sudakov, István Tomon, and Adam~Zsolt Wagner.
\newblock Uniform chain decompositions and applications.
\newblock {\em Random Structures \& Algorithms}, 60(2):261--286, 2022.
\newblock URL: \url{https://doi.org/10.1002/rsa.21034}.

\end{thebibliography}

\end{document}